\documentclass[11pt,a4paper]{article}
\usepackage{a4wide}
\usepackage{cite}
\usepackage[utf8]{inputenc}
\usepackage[T1]{fontenc}

\usepackage{amsmath}
\usepackage{amscd}
\usepackage{amsthm}
\usepackage{amssymb}
\usepackage{cases}

\usepackage{graphicx}
\usepackage{latexsym}
\usepackage{color}
\usepackage{array,multirow,makecell}
\setcellgapes{1pt}
\makegapedcells

\newtheorem{thm}{Theorem}[section]
\newtheorem{prop}[thm]{Proposition}
\newtheorem{lem}[thm]{Lemma} 

\theoremstyle{definition}

\theoremstyle{remark} 
\newtheorem{remark}{Remark}

\newcommand{\TT}{\mathbb{T}^3}

\newcommand{\Div}{\mathrm{div}\,}

\newcommand{\di}{\,\mathrm{d}}
\newcommand{\D}{\,\mathrm{D}\,}
\newcommand{\Dt}{\dfrac{\mathrm{d}}{\mathrm{d} t}}

\newcommand{\ov}{\overline}

\newcommand{\ep}{\varepsilon}

\title{Compression Effects in Heterogeneous Media}
\author{D. Bresch\footnote{Univ. Grenoble Alpes, Univ. Savoie Mont Blanc,  CNRS, LAMA,  Chamb\'ery, France; didier.bresch@univ-smb.fr}, \v S. Ne\v casov\' a \footnote{Institute of Mathematics, Academy of Sciences of the Czech Republic, Prague;  matus@math.cas.cz}, C. Perrin\footnote{Aix Marseille Univ, CNRS, Centrale Marseille, I2M, Marseille, France; charlotte.perrin@univ-amu.fr}\,}

\begin{document}
\maketitle
\begin{small}
	\begin{center}
		{\bf Abstract}
	\end{center}
We study in this paper compression effects in heterogeneous media with maximal packing constraint. Starting from compressible Brinkman equations, where maximal packing is encoded in a singular pressure and a singular bulk viscosity, we show that the global weak solutions converge (up to a subsequence) to global weak solutions of the two-phase compressible/incompressible Brinkman equations with respect to a parameter 
$\varepsilon$ which measures effects close to the maximal packing value. Depending on the importance of the bulk viscosity with respect to the pressure in the dense regimes, memory effects are activated or not at the limit in the congested (incompressible) domain.
	
	\bigskip
	\noindent{\bf Keywords:} Compressible Brinkman Equations; Maximal Packing; Singular Limit; Free Boundary Problem, Memory Effect.
	
	\medskip
	\noindent{\bf MSC:} 35Q35, 35B25, 76T20.
\end{small}

\section*{Introduction}
    We analyze in this paper macroscopic models for heterogeneous media like mixtures, suspensions or crowds, in dense regimes.  These regimes exhibit interesting behaviors such as transition phases with congestion (also called jamming for granular flows) and non-local (in time and/or in space) effects which are both due to a physical packing constraint, that is the finite size of the microscopic components. At the macroscopic scale this packing constraint corresponds to a maximal density constraint $ \rho \leq \rho^*$. A very challenging issue in physics and mathematics is then to model and analyze the change of behavior in congested domains $\rho = \rho^*$ and close to a transition phase $\rho^*-\ep<\rho<\rho^*$.

Two different approaches are generally considered in the literature to model congestion phenomena at the macroscopic level.
The first one, usually called \emph{hard approach}, consists in coupling compressible dynamics in the free domain $\{\rho < \rho^*\}$, with incompressible dynamics in the congested domain $\{\rho = \rho^*\}$. 
Associated to the incompressibility constraint on the velocity field, an additional potential (seen as the Lagrange multiplier) is activated in the congested regions.
The second one which, by opposition, is called \emph{soft approach}, prevents the apparition of congested phases by introducing in the compressible dynamics repulsive forces which become singular as $\rho$ approaches $\rho^*$.
These repulsive effects can be describe either in the pressure (constraint on the fluid at equilibrium) or in the bulk viscosity coefficient, which represents the resistance of the material to a compression.
The interested reader is referred on these two approaches to~\cite{maury2012} and Section~\ref{sec:literature} below for additional references.
An intuitive link can be made between the two approaches: if the scope of action of the repulsive forces tends to $0$, one expects that the soft congestion model degenerates towards a hard congestion model.
We give in the Section~\ref{sec:literature} below some conjectures on this singular limit and recent results that have been obtained in this direction.
In particular, one interesting conjecture made initially by {\sc Lefebvre-Lepot} and {\sc Maury} in~\cite{lefebvre2011} is that a singular bulk viscosity would degenerate in the singular limit towards a (incompressible) pressure and would activate memory effects in the limit congested domain.  

We want to investigate rigorously the link between soft and hard systems, by showing how the choice of the constitutive laws, the pressure and the bulk viscosity as functions of the density in the soft models, impacts the behavior of the limit hard system in congested regions assuming a constant shear viscosity.  
More precisely, the main objective of this paper is to characterize the respective effects of singular pressure and bulk viscosity close to the maximal density constraint in order to understand when memory and pressure effects are activated on the limit hard congestion system. 
To that end, we consider the following three-dimensional soft congestion system (based on compressible Brinkman equations) in $(0,T)\times \TT$:
\begin{subnumcases}{\label{eq:semi-stat-0}}
	\partial_t \rho_\varepsilon + \Div (\rho_\varepsilon u_\varepsilon) = 0,
	  \label{eq:semi-stat-0-mass}\\
	\nabla p_\ep(\rho_\varepsilon) - \nabla(\lambda_\ep(\rho_\varepsilon) \Div u_\varepsilon) - 2\Div(\mu \D(u_\varepsilon)) + r u_\varepsilon
	 = f 
	     \label{eq:semi-stat-0-mom} 
\end{subnumcases}
with $\rho_\varepsilon$ the density satisfying the constraint
\begin{equation}\label{constraints}
0 \le \rho_\varepsilon < 1 \hbox{ a.e. } (t,x) \in [0,T]\times \TT,
\end{equation}
and $u_\ep$ the velocity vector field in the material.
The coefficients $p_\ep$ and $\lambda_\ep$ are respectively the pressure law and the bulk viscosity coefficient, defined in this paper as
\begin{equation} \label{eq:singular-law}
p_\ep(\rho_\varepsilon) = \ep \left(\dfrac{\rho_\varepsilon}{1-\rho_\varepsilon}\right)^\gamma, \quad
\lambda_\ep(\rho_\varepsilon) = \ep \left(\dfrac{\rho_\varepsilon}{1-\rho_\varepsilon}\right)^\beta \quad \text{with}\quad \gamma, \beta > 1,
\end{equation}
while the shear viscosity is assumed to be constant: $\mu>0$.
Finally, $ru_\ep$ with $r>0$ represents the drag and the right-hand term, $f$, is a given external force.
Initially $\rho_\varepsilon\vert_{t=0} = \rho_0^\varepsilon$ with
\begin{equation}\label{hyp:bound-rho0}
0 \le \rho^0_\varepsilon  \le R_\varepsilon < 1 \hbox{ and } R_\varepsilon \to 0        
   \hbox{ when } \varepsilon \to 0, \qquad
   \frac{1}{|\TT|}\int_{\TT} \rho^0_\varepsilon (x)\, dx\le M^0 <1. 
\end{equation}
 Let us encode the effect of the singular bulk viscosity through the following PDE equation
that may be obtained from the mass equation 
$$\partial_t ( \Lambda_\varepsilon (\rho_\varepsilon)) 
   + {\rm div}(\Lambda_\varepsilon(\rho_\varepsilon) u_\varepsilon) 
       = - \lambda_\varepsilon(\rho_\varepsilon) \, {\rm div} u_\varepsilon$$
where 
\begin{equation}
\label{constraint1}
\Lambda_\varepsilon(\rho_\varepsilon) = \rho_\varepsilon \int_0^{\rho_\varepsilon} \lambda_\varepsilon (\tau)/\tau^2 \, d\tau = 
    \rho_\varepsilon \Bigl[\frac{1}{\beta-1} \, \varepsilon^{(1+\gamma-\beta)/\gamma}(p_\varepsilon(\rho_\varepsilon))^{(\beta-1)/\gamma}\Bigr].
\end{equation}

    The main objective now is to understand the asymptotic regime which may be obtained by letting $\varepsilon$ go to zero.  This corresponds to the limit towards the hard approach explained previously.  Let us assume that $(\rho_\varepsilon, u_\varepsilon, p_\varepsilon(\rho_\varepsilon), \Lambda_\varepsilon(\rho_\varepsilon))$ tends to $(\rho, u, p, \Lambda)$. Then we get the following system in $(0,T)\times \TT$: 
\begin{subnumcases}{\label{LimitSystem1}}
	\partial_t \rho + \Div (\rho u ) = 0 \\
	\nabla p - \nabla \Pi - 2\Div(\mu \D(u)) + ru = f  \\
	0\le \rho \le 1 \hbox{ and }  p \ge 0
\end{subnumcases}
where
\begin{equation}\label{LimitSystem2}
 \Pi = - (\partial_t \Lambda+ {\rm div}(\Lambda u)) \hbox{ with }  \Lambda \ge 0.
\end{equation}
We also get the following limit initial data
\begin{equation} \label{inicongestion}
\rho\vert_{t=0} = \rho^0 \in [0,1], \qquad  \Lambda\vert_{t=0} = \Lambda^0
   \qquad  \hbox{ in }~ \TT. 
\end{equation}
It remains now to close the limit system by deriving two constraints.  One of these constraints will result from Equality \eqref{constraint1} depending on the sign of $1+\gamma - \beta$ appearing explicitly in the power of $\ep$.  For the last constraint, different scenarios will be obtained using one of the two following relations
\begin{equation}\label{constraint2}
(1-\rho_\varepsilon) \, p_\varepsilon(\rho_\varepsilon) =\varepsilon^{1/\gamma} \rho_\varepsilon (p_\varepsilon(\rho_\varepsilon))^{(\gamma-1)/\gamma} , 
\end{equation}
or
\begin{equation} \label{constraint3}
(1-\rho_\varepsilon) \, \Lambda_\varepsilon(\rho_\varepsilon) 
    = c(\beta) \varepsilon^{1/(\beta-1)} \rho_\varepsilon^{\beta/(\beta-1)} (\Lambda_\varepsilon(\rho_\varepsilon))^{(\beta-2)/(\beta-1)}.
\end{equation}

\bigskip

 More precisely, passing to the limit in \eqref{constraint1} and \eqref{constraint2}-\eqref{constraint3}, we find the following relations  in addition to System \eqref{LimitSystem1}--\eqref{inicongestion}: 
\begin{itemize}
\item If $1+\gamma -\beta=0$ (memory and pressure effect):
\begin{equation}\label{MPE}
p = (\beta-1) \Lambda \qquad  \hbox{  and }   \qquad   (1-\rho) \, p=0.
\end{equation}
\item If $1+\gamma - \beta <0$  (memory but no pressure effect):
\begin{equation}\label{MPPE}
p = 0 \qquad   \hbox{  and }   \qquad  
           (1-\rho)\, \Lambda= 0.
\end{equation}
\item If $1+\gamma -\beta>0$  (pressure but no memory effect):
\begin{equation} \label{PPME}
        \Lambda=0   
        \qquad \hbox{  and }   \qquad    (1-\rho) \, p = 0.
\end{equation}
\end{itemize}

Observe that this formal analysis could be generalized to more general pressure and bulk viscosity laws than~\eqref{eq:singular-law}, to take into account different (singular) possible behaviors close to the maximal constraint.
The key argument relies here in the comparison between the pressure $p_\ep$ and the coefficient $\Lambda_\ep$ in the vicinity of the maximum density.
Let us emphasize the fact that there is no consensus in physics around the order of singularity of these laws (see for instance~\cite{andreotti2013} or~\cite{coussot2005}). 

\bigskip
Note that it is well known that the compressibility of a fluid may be encoded in the pressure and in the bulk viscosity.
Indeed, incompressible systems may be obtained by letting the Mach number $\mathrm{Ma}$, which appears in the dimensionless Navier-Stokes equations in front of the pressure $\dfrac{1}{\mathrm{Ma^2}}\nabla p(\rho)$, go
to zero (see for instance the works of {\sc Desjardins} et al.~\cite{desjardins1999}, {\sc Lions, Masmoudi}~\cite{lions1998Mach}, {\sc Feireisl, Novotny}~\cite{feireisl2009}). 
But the incompressible equations can be also obtain from a large bulk viscosity limit:
if in the bulk viscosity term $\nabla(\lambda^0\Div u)$ one lets $\lambda^0$ go to $+\infty$ then, formally, $\Div u$ should tend to $0$.
This result has been recently proved by {\sc Danchin} and {\sc Mucha} in~\cite{danchin2017}.  \\
In our paper, the main novelty is to consider both singular pressure and singular bulk viscosity depending on the density which will encode incompressibility of the material at the maximal packing value $\rho^*=1$, assuming the shear viscosity to be constant. 
Below this maximal packing value, the material remains compressible.
Historically, studies on compressible Navier-Stokes system with (non-singular) density dependent bulk viscosity $\lambda(\rho)$ and constant shear viscosity $\mu >0$, start from the beautiful paper~\cite{vaigant1995} by {\sc Kazhikov} and  {\sc Waigant} where they prove global existence of strong solutions in two dimensions with periodic boundary conditions and with no vacuum state if initially no vacuum exists. In their paper, the pressure is assumed $p(\rho) = a \rho^\gamma$, $\mu > 0$ and $\lambda(\rho)= \rho^\beta$ with  $\beta >3$. 
Following this result,  {\sc Perepelitsa} proved in~\cite{perepelitsa2006} the global existence of a weak solution with uniform lower and upper bounds on the density when the initial density is away from vacuum.
Finally, the hypothesis on the coefficient $\beta$ has been recently relaxed with possible vacuum state in~\cite{huli2016} and bounded domains have been considered in~\cite{ducomet2013}.
It would be interesting to investigate the problem for singular bulk viscosity and singular pressure laws for the 3D compressible Navier-Stokes equations but this is not the main objective of our paper.\\
We focus here on Brinkman equations where the total acceleration of the fluid is neglected. A typical application we have in mind is the modeling of flows in porous media.
Brinkman equations are a classical extension of the Darcy equation:
\[
u = -\nabla p + f,
\]
with additional viscous terms, here $\nabla(\lambda\Div u) + 2\Div(\mu \D(u))$. 
In their incompressible version, these equations have been rigorously derived by {\sc Allaire} in~\cite{allaire1991} by homogenization techniques from Navier-Stokes equations in a perforated domain.
His result has been then extended by {\sc Desvillettes} et al.~\cite{desvillettes2008} and {\sc Mecherbet}, {\sc Hillairet}~\cite{mecherbet2018}.
The recent study~\cite{el2017etude} provides some new analysis and numerical results on these equations in the incompressible case. 
These equations may also apply in biology in tumor growth modeling.
The interested reader is referred to the study of {\sc Perthame}, {\sc Vauchelet}~\cite{perthame2015} and the references therein.

\bigskip

The paper will be organized as follows: We will first present the main existence and convergence results,
then we will review mathematical studies that have been realized recently around the subject of congestion problems. In the second section, we present important mathematical
properties linked to the system under consideration and the truncated system we first study.   
Passing to the limit with respect to the parameter of the truncation, $\delta$, we get the global existence of weak solutions for the original system~\eqref{eq:semi-stat-0} at $\ep$ fixed.
It will be then possible to pass to the limit with respect to $\varepsilon$ to recover the hard congestion system \eqref{LimitSystem1}--\eqref{LimitSystem2}
with two additional relations which will be, depending on the parameters $\gamma$ and $\beta$, given by \eqref{MPE} or \eqref{MPPE} or
\eqref{PPME}. 
We will divide the study in two sections depending on the sign of $\gamma-\beta$ which correspond to the dominant pressure regime $\gamma > \beta$ (Section~\ref{sec:pressure}) or dominant bulk regime $\beta \ge \gamma$ (Section~\ref{sec:bulk}).

\section{Main results}
We first prove in the paper the existence of global weak solutions to the soft congestion system \eqref{eq:semi-stat-0} when the pressure and the bulk viscosity are defined by~\eqref{eq:singular-law}.  For simplicity, we assume in addition that 
\begin{equation}\label{hyp:f}
f \in L^2\big(0,T; (L^q( \TT))^3\big) \hbox{ with } q>3.
\end{equation}
    Namely, we obtain the following existence result 
\begin{thm}{\label{thm-ep}}
	Let $\rho_\ep^0$ satisfying condition \eqref{hyp:bound-rho0}.
	Assume in addition that
		\begin{itemize}
			\item if $1< \beta< \gamma$
				\begin{equation}\label{hyp:energy-t0-C1}
				\int_{\TT}{\dfrac{\ep}{(1-\rho^0_\ep)^{\gamma-1}}}\leq E^0 <+\infty
				\end{equation}
			\item if $1 < \gamma \leq \beta$
				\begin{equation}\label{hyp:energy-t0-C2}
				\int_{\TT}{\big(\Lambda_\ep(\rho_\ep^0)\big)^2} 
				 \leq \Lambda^0 < +\infty
				\end{equation}
		\end{itemize}	
	Then there exist global weak solutions of
	 \eqref{eq:semi-stat-0}--\eqref{hyp:bound-rho0}
	 at $\ep$ fixed.
\end{thm}

  The next result justifies the formal derivation of system \eqref{LimitSystem1}--\eqref{inicongestion} respectively with relations \eqref{MPE} if $\gamma= \beta-1$,
 with relations  \eqref{MPPE} if $\gamma<\beta-1$ , relations 
 \eqref{PPME} if $\gamma>\beta-1$. 
  More precisely

\begin{thm}{\label{thm-limit}} As $\ep \rightarrow 0$, there exists a subsequence $(\rho_\ep,u_\ep)$ of global weak solutions of  \eqref{eq:semi-stat-0}--\eqref{hyp:bound-rho0} such that $(\rho_\varepsilon, u_\varepsilon, p(\rho_\varepsilon), \Lambda_\varepsilon(\rho_\varepsilon))$ converges weakly to $(\rho,u, p, \Lambda)$ a global weak solution of the hard congestion system \eqref{LimitSystem1}--\eqref{inicongestion} satisfying two algebraic relations which encode the competition between the singular pressure and bulk viscosity, namely

\smallskip

\noindent I) Memory effect in the congested domain {\rm  ($\gamma \le \beta -1$):} 
   		\begin{itemize}
   			\item For $\gamma < \beta - 1$ {\rm (no pressure effect)}:
   			$$p=0, \qquad  (1-\rho)\Lambda = 0$$   			
   			\item For  $\gamma= \beta - 1$ {\rm (pressure effect)}:
   			$$p= (\beta-1) \Lambda, \qquad (1-\rho)\Lambda = 0.$$
	   	\end{itemize}

\noindent II) No memory effect  in the congested domain {\rm (pressure effect)} 
      {\rm ($\gamma > \beta -1$)}: 
		 $$\Lambda= 0, \qquad \mathrm{spt} \ p \subset \{\rho =1\}.$$	
\end{thm}

\bigskip
\paragraph{Incompressible dynamics in the congested domain.}
 In addition to the limit equations~\eqref{LimitSystem1}--\eqref{inicongestion}, one could have add the incompressibility condition $\Div u = 0$ on the congested sets $\{\rho = 1\}$. 
More precisely, we have the following lemma.
	\begin{lem}[Lemma 2.1 in~\cite{lions1999}]
			Let $u \in L^2(0,T;(H^1(\mathbb{T}))^3)$ and $\rho\in L^2((0,T)\times\mathbb{T}^3)$ such that
			\[
			\partial_t \rho+\Div(\rho u)=0 \quad \text{in} \ (0,T)\times\Omega,\qquad \rho(0)=\rho^0
			\]
			then the following two assertions are equivalent
			\begin{itemize}
				\item[(i)] $\Div u = 0$ a.e on $\{\rho \geq 1\}$ and $0 \leq \rho^0 \leq 1$,
				\item[(ii)] $0 \leq \rho(t,x) \leq 1$ a.e. $(t,x)$.
			\end{itemize}
	\end{lem}

\bigskip
\begin{remark}
	An important issue concerning the limit systems that we obtain, is the regularity of the limit pressure $p$. 
		Through our approximation procedure, the limit pressure $p$, if it is not $0$, is a priori a non-negative measure. 
		If one is able to prove that $p \in L^1((0,T)\times \TT)$, it is thus possible to give a sense a.e. to the product $\rho p$ at the limit and then to the ``exclusion constraint''
		\[
		(1-\rho) p = 0
		\] 
		which is another way to express the activation of the pressure in the congested zones (see system~\eqref{eq:BBCR} written below).
		In fact, this is less the justification the exclusion constraint than the regularity of the pressure which is crucial in the mathematical understanding of partially congested flows.
\end{remark}

\section{Historical remarks}{\label{sec:literature}} 
For reader's convenience, we present below the context of this study and give some historical remarks concerning limits from soft approaches to hard approaches for congestion problems.

\smallskip

\noindent {I -- \it  Derivation from compressible Euler equations.}
A first generic hard congestion model is derived in~\cite{bouchut2000} by {\sc Bouchut} et al. from one-dimensional two-phase gas/liquid flows. The Equations read
\begin{subnumcases}{\label{eq:BBCR}}
\partial_t\rho + \partial_x(\rho u) = 0 \\
\partial_t(\rho u) + \partial_x(\rho u^2) + \partial_x p = 0 \\
0 \leq \rho \leq 1 , ~ (1-\rho)p= 0, ~ p \geq 0
\end{subnumcases}
in which the constraint $(1-\rho)p = 0$, sometimes called ``exclusion constraint'', expresses the activation of the pressure $p$ in the congested phase where $\rho = 1$.
The pressure ensures that the maximal density constraint $\rho^*=1$ is not exceeded.
This system has been then studied theoretically by F. {\sc Berthelin} in~\cite{berthelin2002,berthelin2017} who constructs global weak solutions by means of an approximation with sticky blocks (see~\cite{maury2017} for an associated numerical method). P. {\sc Degond} et al. approximate numerically in~\cite{degond2011,degond2017} the solutions of~\eqref{eq:BBCR} with an appropriate discretization of the soft congestion system
\begin{subnumcases}{\label{eq:BBCR-soft}}
\partial_t\rho + \partial_x(\rho u) = 0 \\
\partial_t(\rho u) + \partial_x(\rho u^2) + \partial_x p_\ep(\rho) = 0 \\
p_\ep(\rho) = \ep \left(\dfrac{\rho}{1-\rho}\right)^\gamma, \quad \gamma > 1
\end{subnumcases}  
Although the rigorous derivation of Equations~\eqref{eq:BBCR} from~\eqref{eq:BBCR-soft} ({\it i.e.} the limit $\ep \rightarrow 0$ in~\eqref{eq:BBCR-soft}) has not been proven theoretically, the authors obtain satisfactory numerical results thanks to smart treatment of the singular pressure $p_\ep$ for small $\ep$. 
Let us also mention on the subject the study~\cite{bresch2017} which addresses the issue of the creation of congested zones in 1D and highlights the multi-scale nature of the problem.

\bigskip

\noindent {II --\it Derivation from compressible Navier-Stokes equations.}

\noindent II -- i) {\it Compressible Navier-Stokes equations with constant viscosities.}
The first justification of the link between a soft congestion system and a hard congestion system is given in~\cite{bresch2014} for the one-dimensional case.
In~\cite{peza2015}, the existence of global weak solutions to the multi-dimensional viscous equations
\begin{subnumcases}{\label{eq:peza-soft}}
\partial_t \rho + \Div(\rho u) = 0 \\
\partial_t(\rho u) + \Div(\rho u \otimes u) + \nabla p_\ep(\rho) - \nabla(\lambda \Div u) - 2\Div(\mu \D(u)) = 0\\
p_\ep(\rho) = \ep \left(\dfrac{\rho}{1-\rho}\right)^\gamma, \quad
    \gamma > 3, \quad 2\mu + \lambda > 0
\end{subnumcases}
is first proven for a fixed $\ep > 0$.
Then, the authors show the weak convergence of these solutions as $\ep \rightarrow 0$ toward global weak solutions of the viscous hard congestion system
\begin{subnumcases}{\label{eq:peza-hard}}
\partial_t \rho + \Div(\rho u) = 0 \\
\partial_t(\rho u) + \Div(\rho u \otimes u) + \nabla p - \nabla(\lambda \Div u) - 2\Div(\mu \D(u)) = 0\\
0 \leq \rho \leq 1 , ~ (1-\rho)p= 0, ~ p \geq 0
\end{subnumcases}

\begin{remark}
	Note that the condition $\gamma > 3$ was assumed in~\cite{peza2015} to prove the existence of global weak solutions to~\eqref{eq:peza-soft}.
	Precisely, it was used to prove the equi-integrability of the approximate truncated pressure $p_{\ep,\delta}(\rho_{\ep,\delta})$ as $\delta \rightarrow 0$ (see details of the truncation process in the next Section).
	It is possible in fact to improve the bound on $\gamma$ and show the existence for $\gamma > \frac{5}{2}$ as it has been done by {\sc Feireisl} et al. in~\cite{feireisl2016}.
\end{remark}

\begin{remark} Originally, P.--L. {\sc Lions} and N. {\sc Masmoudi} in~\cite{lions1999} have obtained the same  viscous  system from the compressible Navier-Stokes equations with constant viscosities and pressure $p(\rho) = a \rho^{\gamma_n}$ letting $\gamma_n \rightarrow +\infty$. 
The same limit has been used by  {\sc Perthame} et al.~\cite{perthame2014} for tumor growth modeling on the basis of the porous medium equation instead of Navier-Stokes equations (see the study of {\sc Vauchelet} and {\sc Zatorska}~\cite{vauchelet2017} in the case of Navier-Stokes equations with additional source term in the mass equation).
In this context the singular limit leads to the Hele-Shaw equations, this problem is sometimes called in the literature ``mesa problem''.
\end{remark}

\medskip

\noindent {II -- ii) \it Compressible Navier-Stokes equation with singular density dependent viscosities.}
In the modeling of immersed granular flows this type of singular limit has enabled to prove in~\cite{perrin2016} the link between the suspension regime and the granular regime which was an open conjecture in physics (see~\cite{andreotti2013}).
Precisely, global weak solutions to the following suspension model
\begin{subnumcases}{\label{eq:suspension}}
\partial_t \rho + \Div(\rho u) = 0 \\
\partial_t(\rho u) + \Div(\rho u \otimes u) + \nabla p_\ep(\rho) - \nabla(\lambda_\ep(\rho) \Div u) \nonumber \\
\hspace{3cm} - 2\Div(\mu_\ep(\rho) \D(u)) + r \rho |u|u= 0
\end{subnumcases}
are proven to exist at $\ep > 0$ fixed for singular viscosities and pressure such that
\begin{equation*}
\mu_\ep(\rho) = \mu_0(\rho + p_\ep(\rho)), \quad p_\ep(\rho) = \ep \left(\dfrac{\rho}{1-\rho}\right)^\gamma, \quad \gamma > 1
\end{equation*}
and $\lambda_\ep$ satisfying a specific relation with the shear viscosity (and thus with the pressure), namely
\begin{equation}\label{eq:mu_lambda}
\lambda_\ep(\rho) = 2(\rho \mu_\ep'(\rho)-\mu_\ep(\rho)).
\end{equation}
Under these hypothesis, the solutions are shown to converge to global weak solutions of
\begin{subnumcases}{\label{eq:granular}}
\partial_t \rho + \Div(\rho u) = 0 \\
\partial_t(\rho u) + \Div(\rho u \otimes u) + \nabla p + \nabla \Pi - 2\mu_0\Div((\rho + p) \D(u)) + r\rho|u|u = 0\\
\partial_t p + u \cdot \nabla p = \dfrac{\Pi}{2\mu_0} \label{eq:memory}\\
0 \leq \rho \leq 1, ~ (1-\rho)p = 0,~ p \geq 0
\end{subnumcases}
where the pressures $p$ and $\Pi$ are respectively the weak limits of $p_\ep(\rho_\ep)$ and $\lambda_\ep(\rho_\ep)\Div u_\ep$. 
The important difference between~\eqref{eq:granular} and~\eqref{eq:peza-hard} is the activation of an additional equation~\eqref{eq:memory} linking the two pressures $p$ and $\Pi$.
It results from the relation~\eqref{eq:mu_lambda} that is imposed at $\ep$ fixed.
Indeed, the conservation of mass and~\eqref{eq:mu_lambda} yield (at least formally)
\[
\partial_t \mu_\ep(\rho_\ep) + \Div(\mu_\ep(\rho_\ep) u_\ep) = - \dfrac{1}{2}\lambda_\ep(\rho_\ep)\Div u_\ep
\]
which gives at the limit~\eqref{eq:memory} due to the incompressibility constraint $\Div u = 0$ that is satisfied in the congested domain.
From a modeling point of view, Equation~\eqref{eq:memory} expresses some memory effects in the congested regions, effects that were first identified by A. {\sc Lefebvre-Lepot} and B. {\sc Maury} in a macroscopic 1D model for ``viscous contact''~\cite{lefebvre2011} (see also~\cite{lefebvre2009} for a microscopic approach). 
From a mathematical point of view, this equation is necessary to close the system and relates $\Pi$, which can be seen as the Lagrange multiplier associated to the constraint $\Div u = 0$ in the congested domain, and $p$ called \emph{adhesion potential} which characterizes the memory effects.
This is thus the singularity of bulk viscosity $\lambda_\ep$ which is responsible for the activation of memory effects in~\eqref{eq:granular}. 

\bigskip
In the present paper, we characterize precisely the respective effects of pressure and bulk viscosity.
At the limit on the hard congestion system, we cover in particular the two cases introduced in~\cite{peza2015} and~\cite{perrin2017} where pressure effects or memory effects are activated.

\section{Structural properties and approximate system}

This section is decomposed in three parts. 
After introducing some important quantities, such as the effective flux, and deriving crucial properties linking the pressure and the bulk viscosity, we present an approximate truncated system which formally degenerates to our original singular system~\eqref{eq:semi-stat-0} as the cut-off parameter tends to $0$.
The last part details how we can construct global weak solutions of the truncated system.  

\subsubsection*{Structural properties, effective flux}
 Let $F$ be the viscous effective flux defined as 
 \begin{equation}\label{df:eff_flux}
 F = (2\mu + \lambda(\rho))\Div u - p(\rho)
 \end{equation}
 and the function $\nu$ from the viscosity coefficients
 \begin{equation}\label{df:nu}
 \nu(\rho) = \dfrac{1}{2\mu + \lambda(\rho)}.
 \end{equation}
 We prove the following Lemma.
\begin{lem} 
Let $(\rho,u)$ satisfying in the weak sense the equations
\begin{subnumcases}{\label{eq:sys-0}}
\partial_t \rho + \Div (\rho u) = 0 \\
\nabla p(\rho) - \nabla(\lambda(\rho) \Div u) - 2\Div(\mu \D(u)) + ru = f
\end{subnumcases}
and denote
$$S:= (-\Delta)^{-1}\Div \big(f -r u\big).$$
Then the following relations hold
\begin{equation} \label{eq:lambda_divu}
<\lambda(\rho)\Div u> = <p(\rho)> - \dfrac{\int_{\TT}{(p(\rho)+ S)\nu(\rho)}}{\int_{\TT}{\nu(\rho)}},
\end{equation}
\begin{equation}\label{eq:F-0}
F= S - \frac{<(p(\rho) +S)\, \nu(\rho)>}{<\nu(\rho)>}.
\end{equation}
\end{lem}

\begin{proof}

Observe first that integration in space of the momentum equation yields
\begin{equation}
r \int_{\TT}{u} = \int_{\TT}{f}.
\end{equation}
Applying the $\Div$ operator to~\eqref{eq:semi-stat-0-mom} we obtain
\begin{equation*}
\Delta F = \Div (f-ru). 
\end{equation*}
Then, denoting $<h> = \displaystyle \dfrac{1}{|\TT|}\int_{\TT}{h(x) \di x}$,
\begin{equation}
F = <F> + S,
\end{equation}
i.e.
\begin{equation}\label{eq:expr_lambdadivu}
(2\mu+ \lambda(\rho)) \Div u - p(\rho) = S  \ + < \lambda(\rho) \Div u - p(\rho) >.
\end{equation}
Let us now characterize the mean value of the effective flux in terms of
the density: rewriting this equation as
\begin{equation*}
\Div u = \nu(\rho)\Big(S \ + <\lambda(\rho)\Div u- p(\rho) > + p(\rho)  \Big)
\end{equation*}
and integrating in space, we arrive at~\eqref{eq:lambda_divu}.
Replacing this expression in \eqref{eq:expr_lambdadivu}, we finally get~\eqref{eq:F-0}.
\end{proof}

\subsubsection*{Approximate system}
We introduce now a cut-off parameter $\delta \leq \delta^0 \in (0,1)$ in order to truncate the singular laws $p_\varepsilon$ and $\nu_\varepsilon$.
Namely, we define the truncated laws
\begin{equation}\label{eq:p_delta}
p_{\ep,\delta}(\rho) =
\begin{cases}
~ \ep \dfrac{\rho^\gamma}{(1-\rho)^\gamma}
    & \text{if} \quad \rho \leq 1-\delta, \\
~ \ep \dfrac{\rho^\gamma}{\delta^\gamma} 
   \quad & \text{if} \quad \rho > 1-\delta
\end{cases}
\end{equation}
\begin{equation}\label{eq:lambda_delta} \hspace{-1.3cm}\text{and} \qquad
\lambda_{\ep,\delta}(\rho) =
\begin{cases}
~ \ep \dfrac{\rho^\beta}{(1-\rho)^\beta} \quad & \text{if} \quad \rho \leq 1-\delta, \\
~ \ep \dfrac{\rho^\beta}{\delta^\beta} \quad & \text{if} \quad \rho > 1-\delta
\end{cases}
\end{equation}
and consider the associated system
\begin{subnumcases}{\label{eq:semi-stat-delta}} 
	\partial_t \rho_{\varepsilon, \delta} 
	   + \Div (\rho_{\varepsilon, \delta} u_{\varepsilon, \delta}) = 0,
	  \label{eq:system-delta-mass}\\
  \nabla p_{\ep,\delta}(\rho_{\ep,\delta}) - \nabla(\lambda_{\ep,\delta}(\rho_{\ep,\delta}) \Div u_{\ep,\delta}) - 2\Div(\mu \D(u_{\ep,\delta})) + r u_{\ep,\delta}
  = f \label{eq:system-delta-momentum}
\end{subnumcases}
with initial data
\begin{equation}
\rho_{\ep,\delta}^0 = \rho_\ep^0.
\end{equation}
Let us first give some properties related to $\nu_{\varepsilon,\delta} (\rho_{\varepsilon, \delta})$.
\begin{lem}{\label{lem:nu_delta}}
	Assume that $\rho^0_\ep$ satisfies~\eqref{hyp:bound-rho0}.
	There exist $C_1,C_2 >0$ which do not depend on $\delta$ or $\ep$ such that 
	\[
	\|\nu_{\ep,\delta}(\rho_{\ep,\delta})\|_{L^\infty_{t,x}}\leq C_1,
	\qquad \int_{\TT}{\nu_{\ep,\delta}(\rho_{\ep,\delta})} \geq C_2.
	\]
\end{lem}

\begin{proof} By definition of $\nu_{\ep,\delta}$~\eqref{df:nu}, we directly get
	\[
	\nu_{\ep,\delta} = \dfrac{1}{2\mu + \lambda_{\ep,\delta}} \leq \dfrac{1}{2\mu} < +\infty.
	\]
	Under the assumption on the initial mass~\eqref{hyp:bound-rho0}${}_2$ (which does not depend on $\delta$ or $\ep$), we have
	\begin{align*}
	M^0 |\TT|
	& \geq \int_{\TT}{\rho_{\ep,\delta}} \\
	& \geq \int_{\TT}{\rho_{\ep,\delta} \mathbf{1}_{\{\rho_{\ep,\delta} \geq \frac{1+M^0}{2}\}} } \\
	& \geq \dfrac{1+M^0}{2} \,\mathrm{meas}\,\left\{\rho_{\ep,\delta} \geq \frac{1+M^0}{2}\right\}.
	\end{align*}
	Now, since $ M^0 < \dfrac{1+M^0}{2 }$, it follows that
	\begin{equation}
	\mathrm{meas}\, \left\{ \rho_{\ep,\delta} \geq \frac{1+M^0}{2} \right\} \leq \dfrac{2M^0}{1+M^0}|\TT| < |\TT|
	\end{equation}
	and
	\begin{equation}
	\mathrm{meas}\, \left\{ \rho_{\ep,\delta} < \frac{1+M^0}{2} \right\}  \geq K > 0
	\end{equation}
	with $K = 1-  \dfrac{2M^0}{1+M^0}$ independent of $\delta$ and $\ep$. 
	Then, 
		\begin{align*}
		\int_{\TT}{\nu_{\ep,\delta}(\rho_{\ep,\delta})}
		&  =  \int_{\TT}{\dfrac{1}{2\mu + \lambda_{\ep,\delta}(\rho_{\ep,\delta})} }\\
		& \geq  \int_{\TT}{\dfrac{1}{2\mu + \ep \frac{\rho_{\ep,\delta}^\beta}{(1-\rho_{\ep,\delta})^\beta}}\mathbf{1}_{\{ \rho_{\ep,\delta} <  \frac{1+M^0}{2} \}} } \\
		& \geq \int_{\TT}{\dfrac{1}{2\mu + \ep \left(\frac{1+M^0}{1-M^0}\right)^\beta }\mathbf{1}_{\{ \rho_{\ep,\delta} < \frac{1+M^0}{2} \}} }\\
		& \geq \dfrac{K}{2\mu + \left(\frac{1+M^0}{1-M^0}\right)^\beta} 
		\end{align*}
	where we have used the fact that $\lambda_{\ep,\delta}(\rho_{\ep,\delta})$ is bounded (uniformly in $\delta$ and $\ep$) when $\rho_{\ep,\delta}$ is far from the singularity.
\end{proof}

\begin{lem}{\label{lambdadivu}} Let us assume $u_{\ep,\delta} \in L^2(0,T,(H^1(\TT))^3)$ and $f \in L^2((0,T) \times \TT)$.
Then, we get for all $(p,q) \in [1,+\infty)^2$:
$$ p_{\varepsilon,\delta} (\rho_{\varepsilon,\delta} ) \in L^p (0,T;L^q( \TT)) 
     \Longrightarrow 	
    \lambda_{\varepsilon, \delta} 
     (\rho_{\varepsilon,\delta}) {\rm div} u_{\varepsilon,\delta} \in L^{\min{(2,p)}}(0,T;L^{\min(2,q)}(\TT))
$$
\end{lem}

\begin{proof} We come back to the formula
$$\lambda_{\varepsilon,\delta} (\rho_{\varepsilon,\delta}) \Div u_{\varepsilon,\delta} =
  -2 \mu {\rm div} u_{\varepsilon,\delta}  
  + p_{\varepsilon,\delta} (\rho_{\varepsilon,\delta} ) + S_{\ep,\delta} 
  - \frac{<(p_{\varepsilon,\delta} (\rho_{\varepsilon,\delta} ) + S_{\ep,\delta})\, \nu_{\varepsilon,\delta} (\rho_{\varepsilon,\delta} )>}{<\nu_{\varepsilon,\delta} (\rho_{\varepsilon,\delta} )>}.
  $$
  It suffices now to use previous lemma to conclude.
 \end{proof}

\subsubsection*{Existence of global solutions to the approximate System}

At $\delta$ fixed, one can construct global weak solutions to the truncated.
system
\begin{thm}\label{approx}
	Let $0 < \delta < \delta^0 = 1-R_\ep$ and $\rho_{\ep,\delta}^0 = \rho_\ep^0$ satisfying~\eqref{hyp:bound-rho0}.
	Let us assume $f \in L^2(0,T; L^{\bar{q}}( \TT))$ with $\bar{q}>3$.
	Then for all $T \in (0, +\infty)$ there exists a global weak solution $(\rho_{\ep,\delta},u_{\ep,\delta})$ to the truncated system~\eqref{eq:system-delta-mass}-\eqref{eq:system-delta-momentum}, i.e.
	\begin{enumerate}
		\item $\rho_{\ep,\delta} \in \mathcal{C}([0,T];L^q(\TT)) \cap L^\infty((0,T) \times \TT)$ for all $q \in [1, +\infty)$, $u_{\ep,\delta}\in L^2(0,T;(H^1(\TT))^3)$;
		\item $(\rho_{\ep,\delta},u_{\ep,\delta})$ satisfies \eqref{eq:system-delta-mass}-- \eqref{eq:system-delta-momentum}
		  in the weak sense:
		\begin{align}\label{eq:weak-delta-mass}
		& \int_0^T\int_{\TT}{\rho_{\ep,\delta} \partial_t \phi} 
		+ \int_0^T\int_{\TT}{\rho_{\ep,\delta} u_{\ep,\delta} \cdot \nabla \phi} \nonumber \\
		& \qquad = \int_{\TT}{\rho_\ep(T) \phi(T)} - \int_{\TT}{\rho_\ep^0 \phi(0)}
		\qquad \forall \ \phi \in \mathcal{C}^1([0,T] \times \TT);
		\end{align}	
		\begin{align}\label{eq:weak-delta-div-mom}
		& - \int_0^T\int_{\TT}{p_{\ep,\delta}(\rho_{\ep,\delta})\Div \psi} 
		+ \int_0^T\int_{\TT}{(2\mu+\lambda_{\ep,\delta}(\rho_{\ep,\delta}))
		  \Div u_{\ep,\delta} \, \Div \psi}  \nonumber \\
		 &   +  \int_0^T\int_{\TT}
		      \mu\,  {\rm curl}\,  u_{\ep,\delta} \cdot  {\rm curl} \psi
		      +  r\int_0^T\int_{\TT}{u_{\ep,\delta} \cdot \psi} \nonumber\\
		  & \hspace{2.5cm}    = \int_{\TT}{f  \cdot \psi}
		\qquad \forall \ \psi \in \mathcal{C}^1([0,T] \times \TT) 
		\end{align}		
		\item The \emph{renormalized continuity equation} holds
		\begin{equation}\label{eq:renormalized-delta}
		\partial_t b(\rho_{\ep,\delta}) + \Div(b(\rho_{\ep,\delta}) u_{\ep,\delta}) 
	+ \big(b_+'(\rho_{\ep,\delta})\rho_{\ep,\delta} 
	          - b(\rho_{\ep,\delta})\big)\Div u_{\ep,\delta} = 0,
		\end{equation}
		for any $b$ piecewise ${\cal C}^1$ and where $b'_+$
		is the right derivative of $b$. 
		\item The energy inequality holds
		\begin{align}\label{eq:energy_delta}
		&\sup_{t\in[0,T]} \int_{\TT} H_{\ep,\delta}(\rho_{\ep,\delta})  
		+ \int_0^T\int_{\TT}{(\frac{3}{2}\mu+\lambda_{\ep,\delta}(\rho_{\ep,\delta}) )(\Div u_{\ep,\delta})^2 } + \frac{\mu}{2} \int_0^T\int_{\TT}{ |{\rm curl}(u_{\ep,\delta})|^2} \nonumber\\
		& \hskip1cm  + r \int_0^T\int_{\TT}{ |u_{\ep,\delta}|^2} 
		   \leq  \int_{\TT}{H_{\ep,\delta}(\rho_\ep^0)} +  \frac{1+C}{2\mu} \|f\|^2_{L^2(0,T;L^p(\TT))}.
		\end{align}
		with $C$ the constant linked to Poincar\'e-Wirtinger inequality 
		$$\|g - <g>\|^2_{L^2(\TT)} \le C \|\nabla g \|^2_{L^2(\TT)}, $$
		and where 
			\begin{equation}\label{eq:H_delta}
	H_{\ep,\delta}(\rho) =
	\begin{cases}
	~ \dfrac{\ep}{\gamma-1} \dfrac{\rho^{\gamma}}{(1-\rho)^{\gamma-1}} \quad & \text{if} \quad \rho \leq 1-\delta \\
	~ \dfrac{\ep}{\gamma-1} \dfrac{\rho^\gamma}{\delta^\gamma} - \dfrac{\ep}{(\gamma - 1) \delta^{\gamma }}(1-\delta)^{\gamma} \rho \quad & \text{if} \quad \rho > 1-\delta.\> 
	\end{cases}
	\end{equation}	
	\end{enumerate}
\end{thm}
Note that defining  $\Lambda_{\ep,\delta}$  by
\begin{equation}\label{df:Lambda_delta}
\Lambda_{\ep,\delta}(\rho) =
\begin{cases}
~ \dfrac{\ep}{\beta-1} \dfrac{\rho^\beta}{(1-\rho)^{\beta-1}} \quad & \text{if} \quad \rho \leq 1-\delta \\
~ \dfrac{\ep}{\beta-1} \dfrac{\rho^\beta}{\delta^\beta} - \dfrac{\ep}{(\beta-1)\delta^{\beta}} (1-\delta)^\beta \rho \quad & \text{if} \quad \rho > 1-\delta
\end{cases}
\end{equation} 
we get  the following renormalized continuity equation in $\mathcal{D}'((0,T)\times \TT)$
\begin{equation}\label{eq:Lambda_delta}
\partial_t \Lambda_{\ep,\delta}(\rho_{\ep,\delta}) + \Div (\Lambda_{\ep,\delta}(\rho_{\ep,\delta}) u_{\ep,\delta}) = - \lambda_{\ep,\delta}(\rho_{\ep,\delta})\Div u_{\ep,\delta}.
\end{equation}
 The existence of global weak solutions to the approximate system, namely Theorem \ref{approx}, follows from a standard procedure.
 For reader's convenience, since our main goal is the study of the singular systems, we just present the idea of the proof.
 The analysis is in fact very similar to the classical case with constant bulk viscosity treated in~\cite{lions1998} Chapter 8.2.
 We construct exactly in the same way the solutions by solving first the system for a regular initial data $\rho^0$ via a fixed point argument. 
 Then, for a general initial density $\rho^0 \in L^\infty(\TT)$, we regularize $\rho^0$ and prove that we can pass to the limit with respect to the parameter of the regularization.
 Compactness arguments are needed to identify the limit quantities, and in particular, we need to prove the strong convergence of the sequence of densities. 
 The arguments to justify this strong convergence are non-standard, and different from the case with a constant bulk viscosity term, but we justify in details this point in Section~\ref{sec:delta1} for the limit $\delta \rightarrow 0$.
 We refer to~\cite{lions1998} for more details.

 \bigskip
 Once we have our global weak solutions $(\rho_{\ep,\delta},u_{\ep,\delta})$, we want to pass to the limit with respect to $\delta$ at $\ep$ fixed to get global existence of weak solutions for the singular PDE. 
 It will be then possible to pass to the limit with respect to $\varepsilon$ to get the congestion systems.\\
 We will divide the study in two sections depending on the sign of $\gamma-\beta$. First, we treat the dominant pressure regime $\gamma > \beta$ (Section~\ref{sec:pressure}), then the dominant bulk viscosity regime $\beta \ge \gamma$ (Section~\ref{sec:bulk}).
 
\section{Dominant pressure regime $\gamma > \beta > 1$}\label{sec:pressure}

\subsection{Existence of weak solutions at $\ep$ fixed}


\subsubsection{Uniform estimates with respect to $\delta$}

First of all, observe that if we consider $\delta$ small enough, $\delta < 1-R_\ep$ ($\ep$ is fixed), we ensure that initially 
\[
H_{\ep,\delta}(\rho_{\ep,\delta}^0) = H_{\ep}(\rho_{\ep}^0) \quad \text{bounded in} \quad L^1(\TT).
\]
Thanks to Theorem~\ref{approx}, the solutions $(\rho_{\ep,\delta},u_{\ep,\delta})$ satisfy the energy estimate~\eqref{eq:energy_delta}, so that $(u_{\ep,\delta})_\delta$ is bounded in $L^2(0,T;(H^1(\TT))^3)$
and $(H_{\ep,\delta}(\rho_{\ep,\delta}))_\delta$ is bounded in $L^\infty(0,T;L^1(\TT))$.
In particular, the control of the internal energy $H_{\ep,\delta}$ leads easily to a control of the density
\begin{equation}\label{bound-rd-Lgamma}
\rho_{\ep,\delta} \quad \text{is bounded in} \quad L^{\infty}(0,T;L^\gamma(\TT)).
\end{equation}
In the following Lemma, we improve the control of the density by using the singularity of $H_{\ep,\delta}$ in $\delta$.
\begin{lem}\label{lem:meas_largerho_delta}
	Let $(\rho_{\varepsilon, \delta}, u_{\varepsilon,\delta})$ be a global weak solution
	of the compressible Brinkman system. Then 
	\begin{equation}\label{eq:est-rhodelta}
	\sup_{t\in [0,T]}	\mathrm{meas} \,\big\{ x\in \TT , \ \rho_{\ep,\delta}(t,x) \geq 1-\delta \big\}  \leq C(\ep) \ \delta^{\gamma -1} .
	\end{equation}
\end{lem}

\begin{proof} The energy $H_{\ep,\delta}$ being defined as~\eqref{eq:H_delta}, recalling that $\gamma>1$, we have
	\begin{align*}
	C \geq \sup_{t\in [0,T]} \int_{\TT}{H_{\ep,\delta}(\rho_{\ep,\delta})}
	& \geq \int_{\TT}{H_{\ep,\delta}(\rho_{\ep,\delta}) \mathbf{1}_{\{ \rho_{\ep,\delta} > 1-\delta \}}} \\
	& \geq C\int_{\TT}{\dfrac{\ep}{\delta^\gamma}\Bigl[
		(\rho_{\varepsilon,\delta}^{\gamma-1} - (1-\delta)^{\gamma-1})
		+ (1-\delta)^{\gamma-1}\delta  \Bigr] \rho_{\ep,\delta}
		\mathbf{1}_{\{ \rho_{\ep,\delta} > 1-\delta \}}} \\
	& \geq  C\int_{\TT}{\dfrac{\ep}{\delta^{\gamma-1} }\Bigl[
		(1-\delta)^{\gamma}  \Bigr] 
		\mathbf{1}_{\{ \rho_{\ep,\delta} > 1-\delta \}}}
	\end{align*}
	which ends the proof.
\end{proof}

\begin{lem}  Let $(\rho_{\varepsilon, \delta}, u_{\varepsilon,\delta})$ be a global weak 
of the compressible Brinkman system~\eqref{eq:semi-stat-delta} with $\gamma > \beta > 1$. Then 
\begin{align*} 
& \|\Lambda_{\varepsilon,\delta}(\rho_{\ep,\delta})\|_{L^1((0,T)\times \TT)} +
     \|p_{\varepsilon,\delta}(\rho_{\ep,\delta})\|_{L^1((0,T)\times \TT)}  \nonumber\\
& \qquad + 
     \|\lambda_{\varepsilon,\delta}(\rho_{\ep,\delta}) {\rm div} u_{\varepsilon, \delta}\|_{L^1((0,T)\times \TT)}
     + \|\lambda_{\varepsilon,\delta}(\rho_{\varepsilon,\delta})\|_{L^{\gamma/\beta}((0,T)\times \TT)}
  \le C_\varepsilon
\end{align*}
where $C_\varepsilon$ does not depend on $\delta$.
\end{lem}

\noindent {\it Proof.}
  Integrating in space~\eqref{eq:Lambda_delta} and using~\eqref{eq:lambda_divu}, we have
	\begin{equation}
	\Dt \int_{\TT}{\Lambda_{\ep,\delta}(\rho_{\ep,\delta})} + \int_{\TT}{p_{\ep,\delta}(\rho_{\ep,\delta})}
	= |\TT|\dfrac{\int_{\TT}{\big(S_{\ep,\delta}+p_{\ep,\delta}(\rho_{\ep,\delta})\big)\nu_{\ep,\delta}(\rho_{\ep,\delta})}}{\int_{\TT}{\nu_{\ep,\delta}(\rho_{\ep,\delta})}}.
	\end{equation}
	Using Lemma~\ref{lem:nu_delta} we can bound the right-hand side
	\begin{equation}\label{eq:dt-Lambda-delta}
	\Dt \int_{\TT}{\Lambda_{\ep,\delta}(\rho_{\ep,\delta})} + \int_{\TT}{p_{\ep,\delta}(\rho_{\ep,\delta})} 
	\leq \dfrac{C_1|\TT|}{C_2}\|S_{\ep,\delta}\|_{L^1} + \dfrac{|\TT|}{C_2}\int_{\TT}{p_{\ep,\delta}(\rho_{\ep,\delta})\nu_{\ep,\delta}(\rho_{\ep,\delta})}
	\end{equation}
	where
	\[
	\|S_{\ep,\delta}\|_{L^1} = \|(-\Delta)^{-1}\Div \big(f-r u_{\ep,\delta}\big) \|_{L^1} 
	\leq C \big(\|f\|_{L^2(0,T; L^{\bar{q}}(\TT))} + \|u_{\ep,\delta}\|_{L^2H^1}\big),
	\]
	and
	\begin{align*}
	p_{\ep,\delta}(\rho_{\ep,\delta})\nu_{\ep,\delta}(\rho_{\ep,\delta})
	& = \dfrac{p_{\ep,\delta}(\rho_{\ep,\delta})}{2\mu + \lambda_{\ep,\delta}(\rho_{\ep,\delta})} \\
	& \leq \dfrac{p_{\ep,\delta}(\rho_{\ep,\delta})}{2\mu} \mathbf{1}_{\{\rho_{\ep,\delta} < M_0 \}}
	+ \dfrac{p_{\ep,\delta}(\rho_{\ep,\delta})}{\lambda_{\ep,\delta}(\rho_{\ep,\delta})}\mathbf{1}_{\{M_0 \leq \rho_{\ep,\delta} < 1-\delta \}} 
	+ \dfrac{p_{\ep,\delta}(\rho_{\ep,\delta})}{\lambda_{\ep,\delta}(\rho_{\ep,\delta})}\mathbf{1}_{\{\rho_{\ep,\delta} \geq 1-\delta \}}.
	\end{align*}
	The first term of the right-hand side is bounded since $\rho_{\ep,\delta}$
	 is far from $1$.
	For the two other terms, which become singular as $\delta \rightarrow 0$, we ensure
	\begin{align*}
	\dfrac{p_{\ep,\delta}(\rho_{\ep,\delta})}{\lambda_{\ep,\delta}(\rho_{\ep,\delta})}\mathbf{1}_{\{M_0 \leq \rho_{\ep,\delta} < 1-\delta \}} \nonumber 
	& \leq \dfrac{C}{(1-\rho_{\ep,\delta})^{\gamma-\beta}}  \mathbf{1}_{\{M_0 \leq \rho_{\ep,\delta} < 1-\delta \}} \nonumber\\
	& \leq C(\ep)H_{\ep,\delta}(\rho_{\ep,\delta}) \mathbf{1}_{\{M_0 \leq \rho_{\ep,\delta} < 1-\delta \}} \label{eq:estim-p-v0}
	\end{align*}
	\begin{align*} 
	\dfrac{p_{\ep,\delta}(\rho_{\ep,\delta})}{\lambda_{\ep,\delta}(\rho_{\ep,\delta})}\mathbf{1}_{\{\rho_{\ep,\delta} \geq 1-\delta \}}
	& \leq  C \dfrac{\rho_{\ep,\delta}^{\gamma-\beta}}{\delta^{\gamma-\beta}}  \mathbf{1}_{\{\rho_{\ep,\delta} \geq 1-\delta \}}  \\
	& \leq  C(\ep)\Bigl[H_{\ep,\delta}(\rho_{\ep,\delta})
	                             + \rho\Bigr] \mathbf{1}_{\{\rho_{\ep,\delta} \geq 1-\delta \}} 
	\end{align*}
	since $\beta \in (1,\gamma)$ and thus $0 < \gamma - \beta < \gamma - 1$ (recall Definition~\eqref{eq:H_delta} of $H_{\ep,\delta}$).
	Using now the control of $H_{\ep,\delta}$ and the fact that the total mass is constant we deduce ($\ep$ is fixed here)
	\[ 
	\Dt \int_{\TT}{\Lambda_{\ep,\delta}(\rho_{\ep,\delta})} + \int_{\TT}{p_{\ep,\delta}(\rho_{\ep,\delta})} 
	\leq C(\ep).
	\]
	To conclude, let us observe that, using that $\beta <\gamma$ and the initial conditions~\eqref{hyp:bound-rho0} and~\eqref{hyp:energy-t0-C1}, we have
	\begin{equation*}
	\int_{\TT}{\Lambda_{\ep,\delta}(\rho_{\ep,\delta}(0,\cdot)) } 
	= \int_{\TT}{\Lambda_{\ep}(\rho_\ep^0) } 
	\leq C .
	\end{equation*}
	Hence, we get from integration in time of~\eqref{eq:dt-Lambda-delta} that
	\begin{equation*}\label{eq:bound-p-delta}
	\big(p_{\ep,\delta}(\rho_{\ep,\delta})\big)_\delta~ \text{is bounded in} ~L^1\big((0,T)\times \TT \big).
	\end{equation*}
  Coming back to Lemma \ref{lambdadivu}, we obtain
\begin{equation*}\label{bound-lambdadivu-delta}
\big(\lambda_{\ep,\delta}(\rho_{\ep,\delta}) \Div u_{\ep,\delta}\big)_\delta ~ \text{bounded in}~ L^1\big((0,T)\times \TT\big).
\end{equation*}
  Note that from the pressure estimate, since $\gamma > \beta$, we deduce that
\begin{equation*}\label{bound-lambda-delta}
\big(\lambda_{\ep,\delta}(\rho_{\ep,\delta})\big)_\delta ~ \text{is bounded in} ~L^{\frac{\gamma}{\beta}}\big((0,T)\times \TT \big).
\end{equation*}

These controls on the pressure and the bulk viscosity can now be used to prove a maximal bound on the density.
\begin{prop}\label{prop:bound_rhodelta}
	The density $\rho_{\ep,\delta}$ is bounded in $L^\infty((0,T)\times \TT)$ uniformly with respect to the cut-off parameter $\delta$.
\end{prop}

\begin{proof}
	The continuity equation can be rewritten with the Lagrangian point of view
	\[
	\dfrac{D}{Dt} \rho_{\ep,\delta} = - \rho_{\ep,\delta} \Div u_{\ep,\delta}
	\]
	where $\dfrac{D}{Dt} = \partial_t + u \cdot \nabla$ denotes the material derivative.
	Using~\eqref{eq:expr_lambdadivu}, we get
	\begin{align*}
	& \dfrac{D}{Dt} \rho_{\ep,\delta} + \rho_{\ep,\delta} \nu_{\ep,\delta}(\rho_{\ep,\delta}) p_{\ep,\delta}(\rho_{\ep,\delta})\\
	& \quad =  - \rho_{\ep,\delta} \nu_{\ep,\delta}(\rho_{\ep,\delta}) S_{\ep,\delta}
	- \dfrac{\rho_{\ep,\delta}\nu_{\ep,\delta}(\rho_{\ep,\delta})}{|\TT|} \int_{\TT}{ \big((\lambda_{\ep,\delta})(\rho_{\ep,\delta}) \Div u_{\ep,\delta} - p_{\ep,\delta}(\rho_{\ep,\delta}) \big) }
	\end{align*}
	and thus by integration in time along the trajectory $X_t(x_0)$ starting at time $t=0$ from $x_0$,
	\begin{align*}
	\rho_{\ep,\delta}(t,X_t(x_0))
	& \leq \rho^0_\ep(x_0) +
	\|\rho_{\ep,\delta}\nu_{\ep,\delta}(\rho_{\ep,\delta})\|_{L^\infty_{t,x}}\|S_{\varepsilon,\delta}\|_{L^1_tL^\infty_{x}} \\
	& \quad + \dfrac{\|\rho_{\ep,\delta}\nu_{\ep,\delta}(\rho_{\ep,\delta})\|_{L^\infty_{t,x}}}{|\TT|} \left\|\int_{\TT}{ \big((\lambda_{\ep,\delta})(\rho_{\ep,\delta}) \Div u_{\ep,\delta} - p_{\ep,\delta}(\rho_{\ep,\delta}) \big) }\right\|_{L^1_t} \\
	& \leq 1 + C(\ep),
	\end{align*}
	where we have used the fact that $\rho_\ep^0$ is bounded by $1$ and $\beta > 1$, so that
	\begin{align*} 
	\rho_{\ep,\delta}\nu_{\ep,\delta}(\rho_{\ep,\delta}) 
	& = \dfrac{\rho_{\ep,\delta}}{2\mu + \lambda_{\ep,\delta}(\rho_{\ep,\delta})} \\
	& \leq  \dfrac{M^0}{2\mu}\mathbf{1}_{\{\rho_{\ep,\delta} < M^0\}} + C(\ep) \rho_{\ep,\delta}^{1-\beta}\mathbf{1}_{\{\rho_{\ep,\delta} \geq M^0\}} \\
	& \leq  C(\ep).
	\end{align*}
	Moreover, thanks to~\eqref{hyp:f}, we ensure that
	\[
	\|S_{\varepsilon,\delta}\|_{L^\infty_x} = \|(-\Delta)^{-1}\Div (f-ru_{\ep,\delta})\|_{L^\infty_x} \leq C \big(\|f\|_{L^{\bar{q}}_x}+ \|u_{\ep,\delta}\|_{L^6_x} \big).
	\]
	since $\bar{q} > 3$.
	This achieves the proof of Proposition~\ref{prop:bound_rhodelta}.
\end{proof}

We now improve a little bit the estimate on the pressure. 
That will ensure that its weak limit is more regular than a measure.
\begin{lem}
	The sequence $(p_{\ep,\delta}(\rho_{\ep,\delta}))_\delta$ is equi-integrable.
\end{lem}
\begin{proof}
	Let us take in~\eqref{eq:weak-delta-div-mom} the test function 
	\[
	\psi = \nabla \Delta^{-1}\big((H_{\ep,\delta}(\rho_{\ep,\delta}))^\alpha - <(H_{\ep,\delta}(\rho_{\ep,\delta}))^\alpha> \big) \quad \text{with} \quad \alpha = \dfrac{\gamma - \beta}{\gamma-1} \in (0,1)
	\]
	where for a periodic function $g$ such that $<g>=0$, $h=(-\Delta)^{-1}g$ is the unique periodic solution of
	\[
	-\Delta h = g\quad \text{in}~~ \TT, \quad <h> = 0. 
	\]
	We have
	\begin{align*}
	& \int_0^T\int_{\TT}{p_{\ep,\delta}(\rho_{\ep,\delta}) \left[\big( H_{\ep,\delta}(\rho_{\ep,\delta}) \big)^{\alpha} - <(H_{\ep,\delta}(\rho_{\ep,\delta}))^\alpha> \right] } \\
	& \quad = \int_0^T\int_{\TT}{\Big(\frac{3}{2} \mu + \lambda_{\ep,\delta}(\rho_{\ep,\delta})\Big)\Div u_{\ep,\delta} \left[\big( H_{\ep,\delta}(\rho_{\ep,\delta}) \big)^{\alpha} - <(H_{\ep,\delta}(\rho_{\ep,\delta}))^\alpha> \right]} \\
	& \qquad  - \int_0^T\int_{\TT}{\big(f-ru_{\ep,\delta}\big) \cdot \psi }
	\end{align*}
	and using the controls resulting from the energy inequality, we obtain
	\begin{align*}
	& \int_0^T\int_{\TT}{p_{\ep,\delta}(\rho_{\ep,\delta}) \left[\big( H_{\ep,\delta}(\rho_{\ep,\delta}) \big)^{\alpha} - <(H_{\ep,\delta}(\rho_{\ep,\delta}))^\alpha> \right] } \\
	& \qquad \leq C + \int_0^T\int_{\TT}{\big(H_{\ep,\delta}(\rho_{\ep,\delta}) \big)^{\alpha}\lambda_{\ep,\delta}(\rho_{\ep,\delta})|\Div(u_{\ep,\delta})| }.
	\end{align*}
	Since $\alpha < 1$ and $H_{\ep,\delta}(\rho_{\ep,\delta})$ is bounded in $L^\infty L^1$,
	we control 
	\[
	\left|\int_0^T\int_{\TT}{p_{\ep,\delta}(\rho_{\ep,\delta}) <(H_{\ep,\delta}(\rho_{\ep,\delta}))^\alpha> }\right|
	\leq \|H_{\ep,\delta}(\rho_{\ep,\delta})\|_{L^\infty L^1}^{\alpha} \|p_{\ep,\delta}(\rho_{\ep,\delta})\|_{L^1L^1}
	\]
	and we get then
	\begin{align*}
	\int_0^T\int_{\TT}{p_{\ep,\delta}(\rho_{\ep,\delta}) \big( H_{\ep,\delta}(\rho_{\ep,\delta}) \big)^{\alpha}}
	& \leq C + \|H_{\ep,\delta}(\rho_{\ep,\delta})\|_{L^\infty L^1}^{\alpha} \|p_{\ep,\delta}(\rho_{\ep,\delta})\|_{L^1L^1} \\
	& \qquad + \int_0^T\int_{\TT}{\big(H_{\ep,\delta}(\rho_{\ep,\delta}) \big)^{\alpha}\lambda_{\ep,\delta}(\rho_{\ep,\delta})|\Div(u_{\ep,\delta})| } \\
	& \leq C + \int_0^T\int_{\TT}{\big(H_{\ep,\delta}(\rho_{\ep,\delta}) \big)^{2\alpha}\lambda_{\ep,\delta}(\rho_{\ep,\delta}) \mathbf{1}_{\{ \rho_{\ep,\delta} \geq M^0 \}} }.
	\end{align*} 
	In the right-hand side we have
	\begin{align*}
	&\big(H_{\ep,\delta}(\rho_{\ep,\delta})\big)^{2\alpha}\lambda_{\ep,\delta}(\rho_{\ep,\delta}) \mathbf{1}_{\{ \rho_{\ep,\delta} \geq M^0 \}}\\
	& \quad = \dfrac{1}{(\gamma-1)^{2\alpha}}\dfrac{\ep^{2\alpha + 1}}{(1-\rho_{\ep,\delta})^{2\alpha (\gamma-1) + \beta}} \ \rho_{\ep,\delta}^{2\alpha\gamma + \beta} \mathbf{1}_{\{ \rho_{\ep,\delta} \geq M^0 \}} \\
	& \quad = \dfrac{1}{(\gamma-1)^{\alpha}}\left(H_{\ep,\delta}(\rho_{\ep,\delta})\right)^{\alpha} p_{\ep,\delta}(\rho_{\ep,\delta}) \ \ep^\alpha \dfrac{\rho_{\ep,\delta}^{\alpha \gamma + \beta}}{(1-\rho_{\ep,\delta})^{\alpha(\gamma-1)+\beta-\gamma}} \mathbf{1}_{\{ \rho_{\ep,\delta} \geq M^0 \}}.
	\end{align*}
	For $\alpha = \dfrac{\gamma - \beta}{\gamma-1}$, using the $L^\infty$ bound on $\rho_{\ep,\delta}$, we obtain
	\begin{align*}
	&\big(H_{\ep,\delta}(\rho_{\ep,\delta})\big)^{2\alpha}\lambda_{\ep,\delta}(\rho_{\ep,\delta}) \mathbf{1}_{\{ \rho_{\ep,\delta} \geq M^0 \}}
	< \ C\ep^\alpha \left(H_{\ep,\delta}(\rho_{\ep,\delta})\right)^{\alpha} p_{\ep,\delta}(\rho_{\ep,\delta}) 
	\end{align*}
	with $\ep< \ep^0$ small enough, so that this term is absorbed in the left-hand side of the previous inequality.
	Finally
	\begin{equation}\label{eq:add_control_p_delta}
	\left(H_{\ep,\delta}(\rho_{\ep,\delta})\right)^{\alpha} p_{\ep,\delta}(\rho_{\ep,\delta}) \quad \text{is bounded in} \quad L^1\big((0,T)\times \TT\big).
	\end{equation}	
	Thanks to the De La Vall\'ee-Poussin criterion (see~\cite{feireisl2009} Theorem 0.8), we deduce the equi-integrability of $(p_{\ep,\delta}(\rho_{\ep,\delta}))_\delta$.
\end{proof}


\subsubsection{Limit $\delta \rightarrow 0$}{\label{sec:delta1}}

\paragraph{First convergence results}
Thanks to the estimates we have just derived, there exists a limit density $\rho_\ep$ such that
\begin{equation}
\rho_{\ep,\delta} \rightharpoonup \rho_\ep \quad \text{weakly-* in} ~ L^\infty\big((0,T)\times \TT \big)
\end{equation}
and, passing to the limit $\delta \rightarrow 0$ in~\eqref{eq:est-rhodelta} we get
\begin{equation}\label{eq:bound_rho_ep}
0 \leq \rho_\ep(t,x) < 1 \quad \text{a.e.} \quad (t,x) \in [0,T]\times \TT.
\end{equation}
In addition, there exists a limit velocity $u_\ep$ such that
\begin{equation}
u_{\ep,\delta} \rightharpoonup u_\ep \quad \text{weakly in} \quad L^2(0,T;(H^1(\TT))^3)
\end{equation}
and, due to the continuity equation, we have
\begin{equation}
\rho_{\ep,\delta}  \rightarrow \rho_\ep \quad \text{in} \quad \mathcal{C}_{\text{weak}}([0,T],L^r(\TT)) \quad \forall\, r \in [1,+\infty).
\end{equation}
To identify the weak limit of the nonlinear term $\rho_\delta u_\delta$ we use the next compensated-compactness Lemma.
\begin{lem}[\cite{lions1998} Lemma 5.1]\label{lem:compensated_compactness}
	Let $(g_n)$, $(h_n)$ be two sequences converging respectively to $g$, $h$ in $L^{r_1}(0,T;L^{q_1}(\TT))$ and $L^{r_2}(0,T;L^{q_2}(\TT))$ where $1\leq r_1,r_2\leq \infty$ and $\frac{1}{r_1} + \frac{1}{r_2} = \frac{1}{q_1} + \frac{1}{q_2}$. Assume in addition that
	\begin{enumerate}
		\item $\partial_t g_n$ is bounded in $L^1(0,T;W^{-m,1}(\TT))$ for some $m$ independent of $n$;
		\item $\|h_n\|_{L^1_tH^s_x}$ is bounded for some $s >0$.
	\end{enumerate}
	Then $g_n h_n$ converges to $gh$ weakly in $\mathcal{D}'((0,T)\times \TT)$.
\end{lem}

\bigskip
We apply the result to $g_\ep = \rho_\ep$, $h_\ep = u_\ep$: we ensure the control of $\partial_t \rho_\ep$ in $L^2\big(0,T;H^{-1}(\TT)\big)$ from the continuity equation, while $\nabla u_{\ep,\delta}$ is bounded in $L^2\big((0,T)\times \TT\big)$ thanks to the energy inequality.
Hence
\begin{equation}
\rho_{\ep,\delta} u_{\ep,\delta} \rightharpoonup \rho_\ep u_\ep  \quad \text{weakly-* in} \quad L^\infty(0,T;L^6(\TT)).
\end{equation}
With the estimates on the pressure we deduce
\begin{equation}
p_{\ep,\delta}(\rho_{\ep,\delta}) \rightharpoonup \ov{p_\ep(\rho)} \quad \text{weakly in} \quad L^{1}\big((0,T) \times \TT\big),
\end{equation}
where $\overline{h}$ denotes the weak limit of the sequence $(h_\delta)_\delta$.
Our next goal is to get the strong convergence of the density $\rho_{\ep,\delta}$ in order to identify the limit of the pressure and bulk viscosity which are non-linear functions of the density.
\paragraph{Convergence a.e. of the density.}
	Thanks to the bounds on $\rho_{\ep,\delta}, u_{\ep,\delta}$, we can pass to the limit in the sense of distributions in the renormalized continuity equation
	\[
	\partial_t \rho_{\ep,\delta}^2 + \Div\big( \rho_{\ep,\delta}^2 u_{\ep,\delta} \big) 
	= - \rho_{\ep,\delta}^2 \Div u_{\ep,\delta}
	\]
	which reads at the limit 
	\[
	\partial_t \ov{\rho_{\ep}^2} + \Div\big( \ov{\rho_{\ep}^2} u_{\ep} \big) 
	= - \ov{\rho_{\ep}^2 \Div u_{\ep}}.
	\]
	On the other hand, the limit density $\rho_\ep \in L^\infty((0,T)\times \TT)$ satisfies the renormalized continuity equation
	\[
	\partial_t \rho_{\ep}^2 + \Div\big( \rho_{\ep}^2 u_{\ep} \big) 
	= - \rho_{\ep}^2 \Div u_{\ep}.
	\]
	Defining $\Psi := \ov{\rho_\ep^2} - \rho_\ep^2 \geq 0$, we have then
	\begin{equation}
	\partial_t \Psi + \Div(\Psi u) = \rho_\ep^2 \Div u_\ep - \ov{\rho_\ep^2 \Div u_\ep} \qquad \text{in} \quad \mathcal{D}'.
	\end{equation}
	By replacing $\Div u_\ep$ by its expression in terms of effective flux and pressure, the previous equation can be rewritten as
	\begin{align}\label{eq:Psi-0}
	\partial_t \Psi + \Div(\Psi u) 
	& = \rho_\ep^2\, \ov{\nu_\ep (\rho_\ep)  F_\ep  } - 	\ov{\rho_\ep^2 \nu_\ep (\rho_\ep)  F_\ep  } \\
	& \quad + \rho_\ep^2\, \ov{ p_\ep(\rho) \nu_\ep(\rho_\ep) } - \ov{ \rho_\ep^2 p_\ep(\rho) \nu_\ep(\rho_\ep) }\nonumber
	\end{align}
	
	\begin{remark}
		In the classical case of constant bulk and shear viscosities $\mu^0$, $\lambda^0$, the previous equation writes
		\[
		\partial_t \Psi + \Div(\Psi u) 
		= \dfrac{1}{2\mu^0 + \lambda^0} \left[\rho_\ep^2\, \ov{F_\ep  } - 	\ov{\rho_\ep^2 F_\ep  } 
		+ \rho_\ep^2\, \ov{ p(\rho)} - \ov{ \rho_\ep^2 p(\rho)}
		\right] 
		\]
		and one can prove some weak compactness property of the effective flux (see~\cite{lions1998} Chapter 5 or~\cite{novotny2004}).
		This property ensures that
		\[
		\rho_\ep^2\, \ov{F_\ep  } =	\ov{\rho_\ep^2 F_\ep} 
		\]
		so that
		\[
		\partial_t \Psi + \Div(\Psi u) = \dfrac{1}{2\mu^0 + \lambda^0} \left(\rho_\ep^2\, \ov{ p(\rho)} - \ov{ \rho_\ep^2 p(\rho)}\right).
		\]
		In the usual case where the pressure is a monotone (increasing) function, independent of the parameter $\delta$, then (see Lemma 3.35 in~\cite{novotny2004})
		\begin{equation}\label{eq:monotone-press}
		\rho_\ep^2\, \ov{ p(\rho)} \leq \ov{ \rho_\ep^2 p(\rho)} 
		\end{equation}
		and 
		\[
		\partial_t \Psi + \Div(\Psi u) \leq 0.
		\]
		We conclude by an integration in space 
		\[
		\Dt \int_{\TT}\Psi \leq 0.
		\]
		Recall that by the convexity of the functional $s \mapsto s^2$ we have $\Psi \geq 0$.
		Hence, if initially $\Psi(0,\cdot) = 0$, we obtain
		\[
		\Psi = 0 \quad \text{a.e.} ~(t,x).
		\]
		This ensures the strong convergence of $(\rho_\delta)_\delta$.
		Note finally that this calculation has been extended by {\sc Feireisl} in~\cite{feireisl2002} to non-monotone pressure that are increasing only from a critical density. In this case, one controls the part where the pressure is non-monotone in such way that a Gronwall inequality can be applied to recover at the end $\Psi = 0$ a.e..
		We will see below that we will have to use with such kind of arguments to prove the strong convergence of the density in case of density dependent bulk viscosities.
		We refer the reader to~\cite{bresch2018} for recent developments on more general non-monotone pressures.
	\end{remark}	
	
	\bigskip
	Our study is original in two ways: first, we have here to deal with a density dependent bulk viscosity, secondly, the pressure (as well as the bulk viscosity) depends on the parameter of approximation $\delta$.
	We begin with proving some similar weak compactness properties satisfied by the effective flux.
	
	\begin{prop}\label{prop:compactness_F}
		We ensure the two following properties
		\begin{equation}\label{prop:compactness_F_1}
		\ov{\left(\dfrac{\rho^2 F_\ep}{2\mu +\lambda_\ep(\rho)} \right)} = \ov{\left(\dfrac{\rho^2}{2\mu +\lambda_\ep(\rho)}\right)}~ \ov{F_\ep}\qquad \text{in} \quad  \mathcal{D}'
		\end{equation}
		\begin{equation}\label{prop:compactness_F_2}
		\ov{\left(\dfrac{F_\ep}{2\mu +\lambda_\ep(\rho)}\right)} = \ov{\left(\dfrac{1}{2\mu +\lambda_\ep(\rho)}\right)}~ \ov{F_\ep}\qquad \text{in} \quad \mathcal{D}' 
		\end{equation}
		where $\ov{g}$ denotes the weak limit of the sequence $(g_\delta)$.
	\end{prop}
	
	\begin{proof}
		The proof of these properties follows again from Lemma~\ref{lem:compensated_compactness} with $h_{\ep,\delta} = F_{\ep,\delta}$ and $g_{\ep,\delta}^1 = \rho_{\ep,\delta}^2 \nu_{\ep,\delta}(\rho_{\ep,\delta})$, $g_{\ep,\delta}^2=\nu_{\ep,\delta}(\rho_{\ep,\delta})$.
		Let us check that we control the time derivative of $g_{\ep,\delta}^1$ and $g_{\ep,\delta}^2$ that satisfy the renormalized continuity equations
		\begin{equation}\label{eq:renorm_gi}
		\partial_t g_{\ep,\delta}^i(\rho_{\ep,\delta}) + \Div\big(g_{\ep,\delta}^i(\rho_{\ep,\delta}) u_{\ep,\delta}\big) + \Big((g_{\ep,\delta}^i)'_+(\rho_{\ep,\delta}) \, \rho_{\ep,\delta}-g_{\ep,\delta}^i(\rho_{\ep,\delta}) \Big) \Div u_{\ep,\delta} = 0
		\end{equation}
		with
		\[ (g_{\ep,\delta}^1)'_+(\rho) = 
		\dfrac{2\rho}{2\mu + \lambda_{\ep,\delta}(\rho)}-\dfrac{(\lambda_{\ep,\delta})'_+(\rho)}{(2\mu + \lambda_{\ep,\delta}(\rho))^2}, \qquad (g_{\ep,\delta}^2)'_+(\rho)
		= -\dfrac{(\lambda_{\ep,\delta})'_+(\rho)}{(2\mu + \lambda_{\ep,\delta}(\rho))^2}.\]
		Recall that
		\[ 
		(\lambda_{\ep,\delta})'_+(\rho) = \begin{cases}
		~ \ep \beta \dfrac{\rho^{\beta-1}}{(1-\rho)^{\beta+1}} \quad & \text{if} \quad \rho < 1-\delta \\
		~ \ep \beta \dfrac{\rho^{\beta-1}}{\delta^\beta}\quad & \text{if} \quad \rho \geq 1-\delta \end{cases}
		\] 
		hence 
		\begin{align*}
		\dfrac{(\lambda_{\ep,\delta})'_+(\rho)}{(2\mu + \lambda_{\ep,\delta}(\rho))^2}
		& \leq  C \mathbf{1}_{\{ \rho_{\ep,\delta} < M^0 \}}
		+ \dfrac{(\lambda_{\ep,\delta})'_+(\rho)}{(\lambda_{\ep,\delta}(\rho))^2}\mathbf{1}_{\{ M^0 \leq \rho_{\ep,\delta} < 1-\delta \}}
		+ \dfrac{(\lambda_{\ep,\delta})'_+(\rho)}{(\lambda_{\ep,\delta}(\rho))^2}\mathbf{1}_{\{ \rho_{\ep,\delta} \geq 1-\delta \}} \\
		& \leq C + \dfrac{C}{\ep}(1-\rho_{\ep,\delta})^{\beta -1}\mathbf{1}_{\{ M^0 \leq \rho_{\ep,\delta} < 1-\delta \}}
		+ \dfrac{C}{\ep}\delta^{\beta}\mathbf{1}_{\{ \rho_{\ep,\delta} \geq 1-\delta \}}\\
		& \leq C(\ep)       
		\end{align*}
		and thus
		$g^i_{\ep,\delta}$, $(g^i_{\ep,\delta})'$ are bounded in $L^2$ uniformly with respect to $\delta$ (but not wrt $\ep$).
		By a Cauchy-Schwarz inequality
		\begin{align*}
		& \left\|\Big((g_{\ep,\delta}^i)'_+(\rho_{\ep,\delta}) \, \rho_{\ep,\delta}-g_{\ep,\delta}^i(\rho_{\ep,\delta}) \Big) \Div u_{\ep,\delta} \right\|_{L^1L^1} \\
		& \quad \leq \|(g_{\ep,\delta}^i)'_+(\rho_{\ep,\delta}) \, \rho_{\ep,\delta}-g_{\ep,\delta}^i(\rho_{\ep,\delta}) \|_{L^2L^2} \|\Div u_{\ep,\delta}\|_{L^2L^2} \leq C  
		\end{align*} 
		and, coming back to Equation~\eqref{eq:renorm_gi}, we get
		\begin{equation*}
		\partial_t g_{\ep,\delta}^1, \, \partial_t g_{\ep,\delta}^2 \quad \text{bounded in} \quad L^2(0,T;W^{-1,1}(\TT)).
		\end{equation*}
		On the other hand, we have
		\[
		h_{\ep,\delta} = S_{\ep,\delta} + <\lambda_{\ep,\delta}(\rho_{\ep,\delta})\Div u_{\ep,\delta} - p_{\ep,\delta}>
		\]
		which is bounded in $L^1(0,T;H^1(\TT))$ since
		\[
		\|\nabla S_{\ep,\delta}\|_{L^2L^2} \leq C \big(\|f\|_{L^2L^{\bar{q}}}+ \|u_{\ep,\delta}\|_{L^2H^1}\big).
		\]
		Applying Lemma~\ref{lem:compensated_compactness} we arrive finally  at~\eqref{prop:compactness_F_1}--\eqref{prop:compactness_F_2}.
	\end{proof}

	\bigskip
	Thanks to these properties, Equation~\eqref{eq:Psi-0} rewrites
	\begin{align}\label{eq:Psi}
	\partial_t \Psi + \Div(\Psi u) 
	& = \rho_\ep^2\, \ov{\nu_\ep (\rho_\ep)  F_\ep  } - 	\ov{\rho_\ep^2 \nu_\ep (\rho_\ep)  F_\ep  } \nonumber\\
	& \quad + \rho_\ep^2\, \ov{ p_\ep(\rho) \nu_\ep(\rho_\ep) } - \ov{ \rho_\ep^2 p_\ep(\rho) \nu_\ep(\rho_\ep) }\nonumber \\
	& = \ov{ F_\ep } \, \left( \rho_\ep^2 \,\ov{\nu_\ep(\rho_\ep)} - \ov{\rho_\ep^2 \nu_\ep(\rho_\ep)} \right) \nonumber \\
	& \quad + \rho_\ep^2\, \ov{ p_\ep(\rho) \nu_\ep(\rho_\ep) } - \ov{ \rho_\ep^2 p_\ep(\rho) \nu_\ep(\rho_\ep) }.
	\end{align}
	Let $b_{\ep,\delta}(s) := p_{\ep,\delta}(s) \nu_{\ep,\delta}(s) = \dfrac{p_{\ep,\delta}(s)}{2\mu +\lambda_{\ep,\delta}(s)}$. 
	As we said before, our constitutive laws depend on the parameter $\delta$ which prevents us from writing directly~\eqref{eq:monotone-press} on $b_{\ep,\delta}$.
	Nevertheless, for all $\delta < \delta^0$, $b_{\ep,\delta^0}(\cdot) \leq b_{\ep,\delta}(\cdot)$ then
	\begin{align*}
	& \rho_\ep^2 \ \ov{\left( \dfrac{p_\ep(\rho_\ep)}{2\mu +\lambda_\ep(\rho_\ep)} \right)} - \liminf_{\delta \rightarrow 0} \dfrac{\rho_{\ep,\delta}^2\, p_{\ep,\delta}(\rho_{\ep,\delta})}{2\mu +\lambda_{\ep,\delta}(\rho_{\ep,\delta})}\\
	& \leq \rho_\ep^2 \  \ov{\left( \dfrac{p_\ep(\rho_\ep)}{2\mu +\lambda_\ep(\rho_\ep)} \right)} - \liminf_{\delta \rightarrow 0} \dfrac{\rho_{\ep,\delta}^2\, p_{\ep,\delta^0}(\rho_{\ep,\delta})}{2\mu +\lambda_{\ep,\delta^0}(\rho_{\ep,\delta})}.
	\end{align*}
	Since moreover $\gamma > \beta$, the function $b_{\ep,\delta^0}$ is increasing for any $\delta^0 > 0$ and it follows that (see for instance Lemma 3.35 from~\cite{novotny2004})
	\begin{align*}
	& \rho_\ep^2 \ \ov{\left( \dfrac{p_\ep(\rho_\ep)}{2\mu +\lambda_\ep(\rho_\ep)} \right)} - \liminf_{\delta \rightarrow 0} \dfrac{\rho_{\ep,\delta}^2\, p_{\ep,\delta}(\rho_{\ep,\delta})}{2\mu +\lambda_{\ep,\delta}(\rho_{\ep,\delta})}\\
	& \leq \rho_\ep^2  \left[ \ov{\left( \dfrac{p_\ep(\rho_\ep)}{2\mu +\lambda_\ep(\rho_\ep)} \right)} -  \ov{\left(\dfrac{ p_{\ep,\delta^0}(\rho_{\ep})}{2\mu +\lambda_{\ep,\delta^0}(\rho_{\ep})}\right)}\right] = 0
	\end{align*}
	where the last equality holds due to the strong convergence in $L^1$ of $\ov{b_{\ep,\delta^0}(\rho)}$ to $\ov{b_{\ep}(\rho)}$ as $\delta^0 \rightarrow 0$ (this follows from the equi-integrability of the
	singular pressure and then from the equi-integrability of $b_\delta$).
	We get then by integration in space
	\begin{align}\label{eq:Psi2}
	\Dt \int_{\TT}{\Psi}~ \leq ~ \int_{\TT}{ |\ov{ F_\ep }| \left| \rho_\ep^2 \,\ov{\nu_\ep(\rho_\ep)} - \ov{\rho_\ep^2 \nu_\ep(\rho_\ep)} \right|} .
	\end{align}
	with 
	\begin{align*}
	\Big|\rho_\ep^2 \nu_{\ep,\delta}(\rho_{\ep,\delta}) - \rho_{\ep,\delta}^2 \nu_{\ep,\delta}(\rho_{\ep,\delta})\Big|
	& =   \nu_{\ep,\delta}(\rho_{\ep,\delta}) \ \big|\rho_{\ep}^2  - \rho_{\ep,\delta}^2 \big| 
	\leq C \ \big|\rho_{\ep}^2  - \rho_{\ep,\delta}^2 \big|  
	\end{align*}
	thanks to Lemma~\ref{lem:nu_delta}.
	Letting $\delta \rightarrow 0$ we get
	\begin{align*}
	\int_{\TT}{ |\ov{ F_\ep }| \left| \rho_\ep^2 \,\ov{\nu_\ep(\rho_\ep)} - \ov{\rho_\ep^2 \nu_\ep(\rho_\ep)} \right|} 
	& \leq C \int_{\TT}{ |\ov{ F_\ep }| \big|\rho_{\ep}^2  - \ov{\rho_\ep^2} \big|} \\
	& \leq C \int_{\TT}{ |\ov{ F_\ep }| \big(\ov{\rho_{\ep}^2}  - \rho_\ep^2 \big)}
	= C \int_{\TT}{|\ov{ F_\ep }|\Psi}.
	\end{align*}
	Since $|\ov{F_\ep}| \in L^1(0,T;L^\infty(\TT))$ and since $\displaystyle  \int_{\TT}{\Psi(0,\cdot)} = 0$ (the initial data does not depend on $\delta$),
	we conclude by Gronwall's Lemma that
	\[
	\Psi = 0 \quad \text{a.e.}
	\]
	There exists then a subsequence such that
	\begin{equation}
	\rho_{\ep,\delta} \longrightarrow  \rho_\ep \quad \text{strongly in} \quad L^q\big((0,T) \times \TT\big) \quad\forall q \in [1, +\infty).
	\end{equation}

\paragraph{Limit of the singular laws.}
From the strong convergence of $\rho_{\ep,\delta}$ and the fact that $\rho_\ep < 1$ a.e. (see ~\eqref{eq:bound_rho_ep}) the following convergences hold
\[
p_{\ep,\delta}(\rho_{\ep,\delta}) \longrightarrow p_\ep(\rho_\ep) \quad \text{strongly in} \quad L^{1}\big((0,T)\times \TT\big)
\]
\[
\sqrt{\lambda_{\ep,\delta}(\rho_{\ep,\delta})}\longrightarrow \sqrt{\lambda_\ep(\rho_\ep)} \quad \text{strongly in} \quad L^{\frac{2\gamma}{\beta}}\big((0,T)\times \TT\big).
\]

\paragraph{Limit in the weak formulation of the equations.} 
We can now pass to the limit in the weak formulation of the mass and momentum equations. 
The only delicate term to deal with is the bulk viscosity term $\lambda_{\ep,\delta}(\rho_{\ep,_\delta}) \Div u_{\ep,\delta}$.
We use the strong convergence of $\sqrt{\lambda_{\ep,\delta}(\rho_{\ep,\delta})}$ towards $\sqrt{\lambda_\ep(\rho_\ep)}$ in $L^2_{t,x}$ combined with the weak convergence in $L^2_{t,x}$ of $\Div u_{\ep,\delta}$ to get 
\[ \sqrt{\lambda_{\ep,\delta}(\rho_{\ep,\delta})} \Div u_{\ep,\delta} \longrightarrow \sqrt{\lambda_\ep(\rho_\ep)} \Div u_\ep \quad \text{weakly in} \quad L^{1}((0,T)\times \TT)  \]
Since $\sqrt{\lambda_{\ep,\delta}(\rho_{\ep,\delta})} \Div u_{\ep,\delta}$ also converges weakly in $L^2((0,T)\times \TT)$ (from the energy estimate), we deduce that the previous convergence holds in $L^2((0,T)\times \TT)$.
Using once again the strong convergence of $\sqrt{\lambda_{\ep,\delta}(\rho_{\ep,\delta})}$ we obtain the weak convergence of the whole bulk viscosity term
\begin{equation}
\lambda_\delta(\rho_{\ep,\delta}) \Div u_{\ep,\delta} =\sqrt{\lambda_{\ep,\delta}(\rho_{\ep,\delta})}\sqrt{\lambda_{\ep,\delta}(\rho_{\ep,\delta})} \Div u_{\ep,\delta} \rightharpoonup \lambda_\ep(\rho_\ep) \Div u_\ep \quad \text{in} \quad L^{1}((0,T)\times \TT).
\end{equation} 
\noindent Finally, the limit $(\rho_\ep, u_\ep)$ is a global weak solution of the system~
\begin{subnumcases}{\label{eq:semi-stat-1-ep}}
\partial_t \rho_\ep + \Div (\rho_\ep u_\ep) = 0 \label{eq:semi-stat-1-ep-mass}\\
\nabla p_\ep(\rho_\ep) - \nabla(\lambda_\ep(\rho_\ep) \Div u_\ep) - 2\Div(\mu \D(u_\ep)) + r u_\ep = f \label{eq:semi-stat-1-ep-mom} \\
0 \leq \rho_\ep < 1 \>  \quad \text{a.e.} ~ \> (0,T)\times \TT.
\end{subnumcases}
In addition, we have the energy inequality
\begin{align}\label{eq:energy_ep}
& \sup_{t\in [0,T]} \int_{\TT}{H_\ep(\rho_\ep)} + \int_0^T\int_{\TT}{(\frac{3}{2}\mu  + \lambda_\ep(\rho_\ep)) |\Div u_\ep|^2 }+ \frac{\mu}{2} \int_0^T\int_{\TT} |{\rm curl} \ u_\ep|^2 \\
& + r \int_0^T\int_{\TT} |u_\ep|^2  \leq ~ \int_{\TT}{H_\ep(\rho^0_\ep)} 
    + \frac{1+C}{2\mu} \|f\|^2_{L^2(0,T;L^{\bar{q}}(\TT))} \nonumber 
\end{align}
where 
\[ 
H_\ep(\rho) = \dfrac{\ep}{\gamma -1} \dfrac{\rho^\gamma}{(1-\rho)^{\gamma-1}}.
\]

\subsection{Congestion limit $\ep \rightarrow 0$}

\subsubsection{Uniform estimates in $\ep$} 
From the energy estimate~\eqref{eq:energy_ep}, we deduce the controls of different quantities uniformly with respect to the parameter $\ep$
\begin{align}
& \big(H_\ep(\rho_\ep)\big)_\ep ~ \text{is bounded in} ~L^\infty(0,T;L^1(\TT)) \label{bound_internalernergy}, \\
& \big(u_\ep\big)_\ep ~ \text{is bounded in} ~ L^2\big(0,T;(H^1(\TT))^2\big) ,
\label{bound_u}\\
& \big(\sqrt{\lambda_\ep(\rho_\ep)}\Div u_\ep\big)_\ep ~ \text{is bounded in} ~ L^2\big((0,T)\times\TT\big).
\end{align}
	\paragraph{Control of the pressure.}
\begin{lem}
	Let $(\rho_\ep,u_\ep)$ be a global weak solution to~\eqref{eq:semi-stat-0} with $\gamma > \beta > 1$. 
	Then, there exists a constant $C > 0$ independent of $\ep$ such that
	\begin{equation}
	\|p_\ep(\rho_\ep)\|_{L^1((0,T)\times \TT)} \leq C.
	\end{equation}
\end{lem}

\medskip
\begin{proof}	
	Let us consider in the weak formulation of the momentum equation the test function
	\[
	\psi = \nabla \Delta^{-1}(\rho_{\ep} - <\rho_{\ep}>).
	\]
	The resulting equation writes
\begin{align*}
& \int_0^T\int_{\TT}{p_{\ep}(\rho_{\ep}) \big(\rho_{\ep} - <\rho_{\ep}>\big)  } \\
& \qquad	= \int_0^T\int_{\TT}{\Big(\frac{3}{2}\mu+ \lambda_{\ep}(\rho_{\ep})\Big)\Div(u_{\ep}) \big(\rho_{\ep} - <\rho_{\ep}>\big)  }
	-  \int_0^T\int_{\TT}{(f-ru_\ep) \cdot \psi}.
\end{align*}	
	Using the controls resulting from the energy inequality and the maximal bound on the density $\rho_\ep$, we get then
	\begin{align*}
	& \int_0^T\int_{\TT}{p_{\ep}(\rho_{\ep}) \big(\rho_{\ep} - <\rho_{\ep}>\big)  } \\
	& ~ \leq C(\|\psi\|_{H^1}) + \int_0^T\int_{\TT}{\sqrt{(2\mu+\lambda_{\ep}(\rho_{\ep}))} \sqrt{(2\mu+\lambda_{\ep}(\rho_{\ep}))}|\Div(u_{\ep})|  } \\
	& ~ \leq C + \dfrac{1}{2\eta}\int_0^T\int_{\TT}{(\frac{3}{2}\mu+\lambda_{\ep}(\rho_{\ep}))|\Div(u_{\ep})|^2} 
	+ \dfrac{\eta}{2} \int_0^T\int_{\TT}{(\frac{3}{2}\mu +\lambda_{\ep}(\rho_{\ep}))} \\
	& ~ \leq C + \dfrac{\eta}{2}\int_0^T\int_{\TT}{\lambda_{\ep}(\rho_{\ep})}
	\end{align*}
	for some $\eta>0$ determined below.
	The integrals of the singular terms can be split into two parts (recall that for low densities the functions $p_\ep$, $\lambda_{\ep}$ are bounded)
	\begin{align*}
	& \int_0^T\int_{\TT}{p_{\ep}(\rho_{\ep}) \big(\rho_{\ep} - <\rho_{\ep}>\big) \mathbf{1}_{\{ \rho_{\ep} \geq \frac{1+<\rho_{\ep}>}{2} \} } } \\
	& ~ \leq C - \int_0^T\int_{\TT}{p_{\ep}(\rho_{\ep}) \big(\rho_{\ep} - <\rho_{\ep}>\big) \mathbf{1}_{\{ \rho_{\ep} < \frac{1+<\rho_{\ep}>}{2} \} } } + \dfrac{1}{2}\int_0^T\int_{\TT}{\lambda_{\ep}(\rho_{\ep}) \mathbf{1}_{\{ \rho_{\ep} < \frac{1+<\rho_{\ep}>}{2}\} } } \\
	& \qquad + \dfrac{\eta}{2}\int_0^T\int_{\TT}{\lambda_{\ep}(\rho_{\ep}) \mathbf{1}_{\{ \rho_{\ep} \geq \frac{1+<\rho_{\ep}>}{2}\} } } \\
	&  ~ \leq C + \dfrac{\eta}{2}\int_0^T\int_{\TT}{\lambda_{\ep}(\rho_{\ep}) \mathbf{1}_{\{ \rho_{\ep} \geq \frac{1+<\rho_{\ep}>}{2}\} } }
	\end{align*}
	Since $\beta < \gamma$, we have in addition that
	\begin{align*}
	\lambda_{\ep}(\rho_{\ep}) \mathbf{1}_{\{ \rho_{\ep} \geq \frac{1+<\rho_{\ep}>}{2}\} } 
	& = \left(\dfrac{1+<\rho_{\ep}>}{1-<\rho_{\ep}>}\right)^{\beta-\gamma}\ep \left(\dfrac{\rho_{\ep,\delta}}{1-\rho_{\ep,\delta}}\right)^{\gamma}  \mathbf{1}_{\{ \rho_{\ep} \geq \frac{1+<\rho_{\ep}>}{2}\}} \\
	& \leq p_{\ep}(\rho_{\ep})  \mathbf{1}_{\{ \rho_{\ep} \geq \frac{1+<\rho_{\ep}>}{2}\} } .
	\end{align*}
	Now, with $\eta = \dfrac{1-M^0}{2}$,
	\begin{align*}
	& \dfrac{1-<\rho_{\ep}>}{2} \int_0^T\int_{\TT}{p_{\ep}(\rho_{\ep})\mathbf{1}_{\{ \rho_{\ep} \geq \frac{1+<\rho_{\ep}>}{2} \} } } \\
	& ~ \leq C +  \dfrac{1-M^0}{4} \int_0^T\int_{\TT}{p_{\ep}(\rho_{\ep}) \mathbf{1}_{\{ \rho_{\ep} \geq \frac{1+<\rho_{\ep}>}{2}\} } }, 
	\end{align*}
	and recalling that we have 
	\[
	\dfrac{1-<\rho_{\ep}>}{2} = \dfrac{1-<\rho_{\ep}^0>}{2} \geq \dfrac{1-M^0}{2}
	\]
	due our assumption~\eqref{hyp:bound-rho0}${}_2$,
	we control the integral of the pressure on the set $\{ \rho_{\ep} \geq \frac{1+<\rho_{\ep}>}{2}\}$
	\[
	 \dfrac{1-M^0}{4}
	\int_0^T\int_{\TT}{p_{\ep}(\rho_{\ep})\mathbf{1}_{\{ \rho_{\ep} \geq \frac{1+<\rho_{\ep}>}{2} \} } } \leq C.
	\]
	Finally, since the pressure is bounded on $\{ \rho_{\ep} < \frac{1+<\rho_{\ep}>}{2}\}$ (far from the singularity), we get
	\begin{equation}\label{eq:bound-pep-L1}
	\big(p_{\ep}(\rho_{\ep})\big)_\ep~ \text{bounded in} ~L^1\big((0,T)\times \TT \big).
	\end{equation}
\end{proof}	

\subsubsection{Limit $\ep\rightarrow 0$}

The boundedness of $p_\ep(\rho_\ep)$ in $L^1$ yields
\begin{equation}
p_\ep(\rho_\ep) \rightarrow p \quad \text{weakly in} \quad \mathcal{M}_+\big((0,T) \times \TT\big).
\end{equation}
Since 
\[
p_\ep(\rho_\ep) \mathbf{1}_{\{\rho_\ep \leq 1-\ep^{1/(\gamma+1)} \}} \leq  \ep^{1/(\gamma+1)}
~\underset{\ep \rightarrow 0}{\longrightarrow} ~ 0  
\]
we deduce that the limit pressure is such that
\begin{equation}\label{eq:spt-p}
\mathrm{spt}\, p \subset \{(t,x),\ \rho = 1\}.
\end{equation}
On the other hand, we have
\begin{align*}
&\lambda_\ep(\rho_\ep) \leq C \ep \mathbf{1}_{\{\rho_\ep \leq M^0\}} 
+ C \ep^{1- \frac{\beta}{\gamma}} \big(p_\ep(\rho_\ep)\big)^{\frac{\beta}{\gamma}} \mathbf{1}_{\{\rho_\ep > M^0\}} 
\end{align*}
with $\beta < \gamma$, so that $\lambda_\ep(\rho_\ep)$ converges strongly to $0$ in $L^{\frac{\gamma}{\beta}}((0,T)\times \TT)$.
Hence, $\sqrt{\lambda_{\ep}(\rho_\ep)}$ converges strongly to $0$ in $L^2$ and 
\begin{equation}
\lambda_\ep(\rho_\ep) \Div u_\ep \rightharpoonup  \ov{\lambda_\ep \Div u} = 0 \quad \text{in}~ \mathcal{D}'.
\end{equation}
The limit system reads in this case
\begin{subnumcases}{\label{eq:pb_limit}}
\partial_t \rho + \Div(\rho u) = 0 \\
\nabla p -2\Div(\mu \D(u)) + ru = f \\
0 \leq \rho \leq 1, \quad \mathrm{spt}\, p \subset \{ \rho = 1\}, \quad p \geq 0 
\end{subnumcases}
where the memory effects are never activated.

\section{Dominant bulk viscosity regime $1<\gamma\leq \beta$}\label{sec:bulk}
Let us now consider the case where $\beta \geq \gamma$. 
If the approach proceeds formally in the same way (regularization of the system by truncation of the singular laws and study of the behavior as $\ep \rightarrow 0$), we have here to adapt the arguments to get the appropriate uniform controls in $\delta$ and $\ep$. 
For that purpose, we shall distinguish in the estimates three cases: $\gamma < \beta-1$, $\gamma = \beta -1$ and $\beta-1 \leq \gamma < \beta$ that correspond to the sub-cases presented in Theorem~\ref{thm-limit}.
In these three cases, we are not able to control the bulk viscosity coefficient $\lambda_\ep$ from the pressure $p_\ep$.
Nevertheless, we will see that Equation~\eqref{eq:expr_lambdadivu} enables to pass to the limit $\ep \rightarrow 0$ in the two cases $\gamma \in (\beta-1, \beta]$ (no memory effects at the limit), $\gamma \leq \beta-1$ (memory effects).

\subsection{Case $\beta-1 < \gamma \leq \beta$, $\gamma > 1$} 

\subsubsection{Existence of weak solutions at $\ep$ fixed}
We consider the same regularized system~\eqref{eq:semi-stat-delta} with truncated pressure~\eqref{eq:p_delta} and bulk viscosity~\eqref{eq:lambda_delta} as in the previous section.
Recall that we ensure from Lemma~\ref{lem:nu_delta} the following properties on $\nu$:
\begin{equation}
\exists\, C_1, C_2 > 0 \quad \text{s.t.} \quad \|\nu_{\ep,\delta}(\rho_{\ep,\delta})\|_{L^\infty_{t,x}}\leq C_1,
\quad \int_{\TT}{\nu_{\ep,\delta}(\rho_{\ep,\delta})} \geq C_2.
\end{equation}
\begin{lem}\label{lem:controls-1}  Let $(\rho_{\varepsilon, \delta}, u_{\varepsilon,\delta})$ be a global weak solution of the compressible Brinkman system~\eqref{eq:semi-stat-delta} with $\gamma \leq \beta$. Then, 
	\begin{align} 
	& \|\Lambda_{\varepsilon,\delta}(\rho_{\ep,\delta})\|_{L^1((0,T)\times \TT)} +
	\|p_{\varepsilon,\delta}(\rho_{\ep,\delta})\|_{L^1((0,T)\times \TT)}   + 
	\|\lambda_{\varepsilon,\delta}(\rho_{\ep,\delta}) {\rm div} u_{\varepsilon, \delta}\|_{L^1((0,T)\times \TT)}
	\le C
	\end{align}
	where $C$ does not depend on $\delta$ or $\ep$.
\end{lem}
\begin{proof} Starting again from the equation on $\Lambda_{\ep,\delta}$~\eqref{eq:Lambda_delta}, using~\eqref{eq:lambda_divu} to replace $\lambda_{\ep,\delta} \Div u_{\ep,\delta}$, and integrating in space we have
\begin{align*}{\label{eq:control-Lambda}}
& \dfrac{\di}{\di t} \int_{\TT}{\Lambda_{\ep,\delta}(\rho_{\ep,\delta})} + \int_{\TT}{p_{\ep,\delta}(\rho_{\ep,\delta})} \nonumber \\
&  \leq C \Big(\|S_{\ep,\delta}\|_{L^1} + \int_{\TT}{p_{\ep,\delta}(\rho_{\ep,\delta}) \nu_{\ep,\delta}(\rho_{\ep,\delta}) \mathbf{1}_{\{\rho_{\ep,\delta} \leq M^0\}}}
+ \int_{\TT}{p_{\ep,\delta}(\rho_{\ep,\delta}) \nu_{\ep,\delta}(\rho_{\ep,\delta}) \mathbf{1}_{\{M^0 < \rho_{\ep,\delta} \}}} \Big)\nonumber \\
&  \leq C + C \int_{\TT}{\mathbf{1}_{\{\rho_{\ep,\delta} \leq M^0\}}} 
+ \int_{\TT}{ \dfrac{C}{(1-\rho_{\ep,\delta})^{\gamma - \beta}}\mathbf{1}_{\{M^0 < \rho_{\ep,\delta} \leq 1-\delta \}}}
+  \int_{\TT}{ \dfrac{C}{\delta^{\gamma - \beta}}\mathbf{1}_{\{\rho_{\ep,\delta} > 1-\delta\}} } \\
\end{align*}
for some constant $C$ independent of $\delta, \ep$.
Now, since $\gamma < \beta$, we can bound uniformly in $\delta,\ep$ the right-hand side and therefore
\begin{equation*}
\big(\Lambda_{\ep,\delta}(\rho_{\ep,\delta})\big) \quad \text{is bounded in} \quad L^\infty\big(0,T;L^1(\TT)\big),
\end{equation*} 
\[
\big(p_{\ep,\delta}(\rho_{\ep,\delta})\big) \quad \text{is bounded in} \quad L^1\big((0,T)\times \TT\big).
\]
We have as a byproduct (Lemma~\ref{lambdadivu})
\begin{equation*}\label{eq:lambda-delta-2}
\big(\lambda_{\ep,\delta}(\rho_{\ep,\delta}) \Div u_{\ep,\delta}\big)~ \text{bounded in}~ L^1\big((0,T)\times \TT\big).
\end{equation*}
\end{proof}

\bigskip
\noindent
\textit{Uniform upper bound on the density.}
Proposition~\ref{prop:bound_rhodelta} still holds in the case $\gamma < \beta$: there exists $C$ which does not depend on $\delta$ such that
\begin{equation}\label{eq:bound_rhodelta_2}
\|\rho_{\ep,\delta}\|_{L^\infty_{t,x}} \leq \bar{\rho} < \infty.
\end{equation}

\medskip
\begin{lem}\label{lem:bound-rhod-Lambda}
	For any $t \in [0,T]$, we have 
	\begin{equation}\label{eq:est-rho-delta-2}
	\mathrm{meas} \,\big\{ x\in \TT , \ \rho_{\ep,\delta}(t,x) \geq 1-\delta \big\}  \leq C(\ep) \ \delta^{\beta-1} .
	\end{equation}
\end{lem}

\begin{proof} 	
		Since initially, we assume~\eqref{hyp:energy-t0-C2}, we can mimic the proof of~\eqref{eq:est-rhodelta} by using $\Lambda_{\ep,\delta}$ instead of $H_{\ep,\delta}$
		\begin{align*}
		C \geq \sup_{t\in [0,T]} \int_{\TT}{\Lambda_{\ep,\delta}(\rho_{\ep,\delta})}
		& \geq \int_{\TT}{\Lambda_{\ep,\delta}(\rho_{\ep,\delta}) \mathbf{1}_{\{ \rho_{\ep,\delta} > 1-\delta \}}} \\
		& \geq C\int_{\TT}{\dfrac{\ep}{\delta^\beta}\Bigl[
			(\rho_{\varepsilon,\delta}^{\beta-1} - (1-\delta)^{\beta-1})
			+ (1-\delta)^{\beta-1}\delta  \Bigr] \rho_{\ep,\delta}
			\mathbf{1}_{\{ \rho_{\ep,\delta} > 1-\delta \}}} \\
		& \geq  C\int_{\TT}{\dfrac{\ep}{\delta^{\beta-1} }\Bigl[
			(1-\delta)^{\beta}  \Bigr] 
			\mathbf{1}_{\{ \rho_{\ep,\delta} > 1-\delta \}}}.
		\end{align*}
\end{proof}

As explained in the previous section, passing to the limit in the weak formulation of the equations requires additional estimates on the singular laws. 

\bigskip
\begin{lem}\label{lem:controls-2}Let $(\rho_{\varepsilon, \delta}, u_{\varepsilon,\delta})$ be a global weak solution of the compressible Brinkman system~\eqref{eq:semi-stat-delta} with $\gamma \in (\beta-1,\beta]$.
If initially~\eqref{hyp:energy-t0-C2} is satisfied, then there exist two constants $C^1>0$ independent of $\delta,\ep$, and $C^2_\ep> 0$ independent of $\delta$, such that
\begin{equation}
\|\Lambda_{\ep,\delta}(\rho_{\ep,\delta})\|_{L^\infty(0,T; L^2(\TT))} 
\leq C^1,
\end{equation}
\begin{equation}
 \|p_{\ep,\delta}(\rho_{\ep,\delta})\|_{L^{1+\frac{\beta-1}{\gamma}}((0,T)\times \TT)} + \|\lambda_{\ep,\delta}(\rho_{\ep,\delta})\|_{L^{\frac{\beta + \gamma -1}{\beta}}((0,T)\times \TT)} 
\leq C^2_\ep.
\end{equation}
\end{lem}

\bigskip
\begin{proof}
	We multiply Equation~\eqref{eq:Lambda_delta} by $\Lambda_{\ep,\delta}(\rho_{\ep,\delta})$:
	\begin{align}\label{eq:eta-Lambda}
	& \partial_t\big(\Lambda_{\ep,\delta}(\rho_{\ep,\delta})\big)^2 + \Div\big( \big(\Lambda_{\ep,\delta}(\rho_{\ep,\delta})\big)^2 u_{\ep,\delta} \big) \nonumber\\
	& = - \Big[\rho_{\ep,\delta}  (\Lambda_{\ep,\delta})'_+(\rho_{\ep,\delta})\Lambda_{\ep,\delta}(\rho_{\ep,\delta})
	+ \Lambda_{\ep,\delta}(\rho_{\ep,\delta}) \lambda_{\ep,\delta}(\rho_{\ep,\delta})
	\Big] \Div u_{\ep,\delta}  \nonumber \\
	& = - \left[\frac{\rho_{\ep,\delta} (\Lambda_{\ep,\delta})'_+(\rho_{\ep,\delta})}{\lambda_{\ep,\delta}(\rho_{\ep,\delta})} + 1\right] 
	 \Lambda_{\ep,\delta}(\rho_{\ep,\delta})
	 \lambda_{\ep,\delta}(\rho_{\ep,\delta})\, {\rm div} u_{\varepsilon, \delta} 
	\end{align}
	where
	\[
	(\Lambda_{\ep,\delta})'_+(\rho) = 
	\begin{cases}
	~ \dfrac{\ep}{(1-\rho)^\beta} \dfrac{(\beta-\rho)\rho^{\beta-1}}{\beta-1} \quad&\text{if}\quad \rho < 1-\delta \\
	~ \dfrac{\ep}{\delta^\beta} \dfrac{\beta\rho^{\beta-1} - (1-\delta)^\beta}{\beta-1} \quad&\text{if}\quad \rho \geq 1-\delta 
	\end{cases}
	\]
	which is such that
	\[
	0 \leq \dfrac{\rho_{\ep,\delta} \ (\Lambda_{\ep,\delta})'_+(\rho_{\ep,\delta})}{\lambda_{\ep,\delta}(\rho_{\ep,\delta})}
	\leq \dfrac{\beta}{\beta-1}.
	\]
	Recall now that 
	\begin{align*}
	\lambda_{\ep,\delta}(\rho_{\ep,\delta})\Div u_{\ep,\delta}
	& =  p_{\ep,\delta} (\rho_{\ep,\delta})
	+  S
	- 2\mu \, \Div u_{\ep,\delta}  \\
	& \quad  + \dfrac{1}{|\TT|} \int_{\TT}{\big( \lambda_{\ep,\delta}(\rho_{\ep,\delta})\Div u_{\ep,\delta} -  p_{\ep,\delta} (\rho_{\ep,\delta}) \big)}
	\end{align*}
	that we can replace in~\eqref{eq:eta-Lambda}.
	Integrating next in space, we obtain
	\begin{align}\label{eq:eta-2}
	& \Dt \int_{\TT}{\big(\Lambda_{\ep,\delta}(\rho_{\ep,\delta})\big)^{2}} 
	+ \dfrac{2\beta -1}{\beta-1} \int_{\TT}{\Lambda_{\ep,\delta}(\rho_{\ep,\delta}) p_{\ep,\delta} (\rho_{\ep,\delta})}  \nonumber \\
	&\quad = 
	- \dfrac{2\beta -1}{\beta-1}\int_{\TT}{\Lambda_{\ep,\delta}(\rho_{\ep,\delta}) \ S_{\ep,\delta}}
    + 2\mu\dfrac{2\beta -1}{\beta-1} \int_{\TT}{\Lambda_{\ep,\delta}(\rho_{\ep,\delta}) \Div u_{\ep,\delta} } \nonumber \\
	& \qquad 
	- \dfrac{2\beta -1}{\beta-1} \dfrac{1}{|\TT|}\left(\int_{\TT}{\Lambda_{\ep,\delta}(\rho_{\ep,\delta})}\right) \left( \int_{\TT}{\big( \lambda_{\ep,\delta}(\rho_{\ep,\delta})\Div u_{\ep,\delta} -  p_{\ep,\delta} (\rho_{\ep,\delta}) \big)}	\right) \nonumber \\
	& \quad= I_1 + I_2 + I_3 .\nonumber       
	\end{align}
	By a Cauchy-Schwarz inequality, we have
	\begin{align*}
	|I_1| + |I_2|
	& \leq C (\|S_{\ep,\delta}\|_{L^2}^2 + \|\Div u_{\ep,\delta}\|_{L^2}^2) + \int_{\TT}{\big(\Lambda_{\ep,\delta}(\rho_{\ep,\delta})\big)^{2}}
	\end{align*}
	while, due to the previous estimates, 
	\begin{align*}
	\int_0^T|I_3| 
	& \leq C \|\Lambda_{\ep,\delta}(\rho_{\ep,\delta})\|_{L^\infty L^1}
	       \big(\|\lambda_{\ep,\delta}(\rho_{\ep,\delta})\Div u_{\ep,\delta}\|_{L^1L^1} +  \|p_{\ep,\delta} (\rho_{\ep,\delta})\|_{L^1L^1} \big).
	\end{align*}
	We conclude then with a Gronwall inequality (recall that initially we assume~\eqref{hyp:energy-t0-C2}) and get
	\begin{equation}
	\sup_{t \in [0,T]}\int_{\TT}{\big(\Lambda_{\ep,\delta}(\rho_{\ep,\delta})\big)^{2}} 
	+ \dfrac{2\beta -1}{\beta-1} \int_{\TT}{\Lambda_{\ep,\delta}(\rho_{\ep,\delta}) p_{\ep,\delta} (\rho_{\ep,\delta})} 
	\leq C + C\int_{\TT}{\big(\Lambda_{\ep,\delta}(\rho_{\ep}^0)\big)^{2}}.
	\end{equation}
With the control of $\Lambda_{\ep,\delta}(\rho_{\ep,\delta}) p_{\ep,\delta}(\rho_{\ep,\delta})$ in $L^1$, we deduce that
\[
\left\|\dfrac{1}{(1-\rho_{\ep,\delta})^{\gamma + \beta-1} }\mathbf{1}_{\{ \rho_{\ep,\delta} \leq 1-\delta \}} \right\|_{L^1} \leq C(\ep) 
\]
and
\[
\left\|\dfrac{\rho_{\ep,\delta}^{\gamma + \beta -1}}{\delta^{\gamma+ \beta-1}} \mathbf{1}_{\{ \rho_{\ep,\delta} > 1-\delta \}} \right\|_{L^1} \leq C(\ep) 
\]
since (see Lemma~\ref{lem:bound-rhod-Lambda} and~\eqref{eq:bound_rhodelta_2})
\[
\Lambda_{\ep,\delta}(\rho_{\ep,\delta})\mathbf{1}_{\{ \rho_{\ep,\delta} > 1-\delta \}}
\geq \dfrac{\ep (1-\delta)^\beta}{\delta^{\beta-1}}
\quad \text{and} \quad
\|\rho_{\ep,\delta}\|_{L^\infty} \leq C.
\]
Hence,
\begin{equation}\label{eq:pdelta2}
\big(p_{\ep,\delta}(\rho_{\ep,\delta})\big)_\delta  ~ \text{is bounded in} ~ L^{1 + \frac{\beta-1}{\gamma}}((0,T)\times \TT) ,
\end{equation}
\begin{equation}\label{eq:lambdadelta2}
\big(\lambda_{\ep,\delta}(\rho_{\ep,\delta})\big)_\delta  ~ \text{is bounded in} ~ L^{\frac{\gamma + \beta - 1}{\beta}}((0,T)\times \TT).
\end{equation}
We emphasize the fact that these controls are not uniform in $\ep$. 
This is due to the fact that the pressure $p_\ep$ and the bulk viscosity $\lambda_\ep$ are ``more singular'' than $\Lambda_\ep$ close to $1$ ($\gamma, \beta > \beta-1$).
\end{proof}

\begin{remark}
	We assumed in this case that $\gamma > 1$ which enables, thanks to~\eqref{eq:lambdadelta2}, to bound $\lambda_{\ep,\delta}(\rho_{\ep,\delta})$ in $L^1((0,T) \times \TT)$.
\end{remark}

\bigskip
\subsubsection*{Limit $\delta \rightarrow 0$}
We can pass to the limit in the weak formulation of the equation except in the non-linear terms.
As in the previous section, we need to prove the strong convergence of $\rho_{\ep,\delta}$.
Note first that the results of Proposition~\ref{prop:compactness_F} still hold:
\begin{equation}\label{prop:compactness_F_1-2}
\ov{\left(\dfrac{\rho^2 F_\ep}{2\mu +\lambda_\ep(\rho)} \right)} = \ov{\left(\dfrac{\rho^2}{2\mu +\lambda_\ep(\rho)}\right)}~ \ov{F_\ep}\qquad \text{in} \quad  \mathcal{D}',
\end{equation}
\begin{equation*}\label{prop:compactness_F_2-2}
\ov{\left(\dfrac{F_\ep}{2\mu +\lambda_\ep(\rho)}\right)} = \ov{\left(\dfrac{1}{2\mu +\lambda_\ep(\rho)}\right)}~ \ov{F_\ep}\qquad \text{in} \quad \mathcal{D}'. 
\end{equation*}
Hence, for $\Psi = \ov{\rho_\ep^2} - \rho_\ep^2 \geq 0$,
\begin{align*}\label{eq:Psi-2}
\partial_t \Psi + \Div(\Psi u) 
& = \ov{ F_\ep } \, \left( \rho_\ep^2 \,\ov{\nu_\ep(\rho_\ep)} - \ov{\rho_\ep^2 \nu_\ep(\rho_\ep)} \right) \nonumber \\
& \quad + \rho_\ep^2\, \ov{ p_\ep(\rho_\ep) \nu_\ep(\rho_\ep) } - \ov{ \rho_\ep^2 p_\ep(\rho_\ep) \nu_\ep(\rho_\ep) }.
\end{align*}
where, in this case, the function $s \mapsto b_{\ep,\delta}(s) = p_{\ep,\delta}(s) \nu_{\ep,\delta}(s)$ is non-monotone.
In fact it is not a problem, because $b_{\ep,\delta}$ is now bounded and we can then treat it as the first part of the right-hand side.
We obtain similarly
\begin{equation}
\dfrac{\di}{\di t} \int_{\TT}{\Psi(t,\cdot)}
\leq C  \int_{\TT}{(|\ov{F_\ep}|+1)\Psi(t,\cdot)}
\end{equation}
with $|\ov{F_\ep}|\in L^1(0,T;L^\infty(\TT))$ which yields again $\Psi = 0$ by Gronwall's inequality.
The strong convergence of the density is thus preserved in this case.
In addition, due to~\eqref{eq:est-rho-delta-2}, we ensure that the limit density satisfies
\begin{equation}\label{eq:rho-ep-max}
0\leq \rho_{\ep} < 1 \quad \text{a.e.}
\end{equation}
We can then pass to the limit in all the terms of the equations. 
In particular, for the bulk viscosity term, the strong convergence of $\rho_{\ep,\delta}$ and~\eqref{eq:rho-ep-max} imply that $\lambda_{\ep,\delta}(\rho_{\ep,\delta})$ converges strongly to $\lambda_\ep(\rho_\ep)$ in $L^{\frac{\gamma + \beta -1}{\beta}}((0,T)\times \TT)$.
As in the previous section we deduce that 
\[
\lambda_{\ep,\delta}(\rho_{\ep,\delta})\Div u_{\ep,\delta}
= \sqrt{\lambda_{\ep,\delta}(\rho_{\ep,\delta})}\sqrt{\lambda_{\ep,\delta}(\rho_{\ep,\delta})}\Div u_{\ep,\delta}
\]
converges weakly to $\lambda_\ep(\rho_\ep)\Div u_\ep$ in $L^q((0,T)\times \TT)$ for some $q>1$.\\
Finally, we pass to the limit in~\eqref{eq:Lambda_delta} (recall $\Lambda_{\ep,\delta}(\rho_{\ep,\delta})$ is bounded in $L^\infty(0,T;L^2(\TT))$ so that $\Lambda_{\ep,\delta}(\rho_{\ep,\delta}) u_{\ep,\delta}$ is bounded in $L^2(0,T;(L^{3/2}(\TT))^3)$):
\begin{equation}\label{eq:Lambda_ep-2}
\partial_t \Lambda_\ep(\rho_\ep) + \Div\big(\Lambda_\ep(\rho_\ep) u_\ep \big) = -\lambda_\ep(\rho_\ep)\Div u_\ep
\end{equation}

\subsubsection{Limit $\ep \rightarrow 0$.} 
At this stage, we control uniformly 
\begin{align*}
&\big(\rho_\ep \big)_\ep \quad  \text{in} \quad L^\infty((0,T)\times \TT), \\
&\big(u_\ep \big)_\ep \quad  \text{in} \quad L^2(0,T;(H^1(\TT))^3), \\
&\big(\Lambda_\ep(\rho_\ep)\big)_\ep \quad  \text{in} \quad L^\infty(0,T; L^2(\TT)) ,\\
&\big(p_\ep(\rho_\ep)\big)_\ep \quad  \text{in} \quad L^1((0,T)\times \TT) ,\\
&\big(\lambda_\ep(\rho_\ep) \Div u_\ep\big)_\ep \quad  \text{in} \quad L^1((0,T)\times \TT).
\end{align*}
Therefore
\begin{equation*}
p_\ep(\rho_\ep) \rightarrow p \quad \text{in}~\mathcal{M}_+((0,T)\times \TT)
\end{equation*}
and
\begin{equation*}
\lambda_\ep(\rho_\ep) \Div u_\ep\rightarrow - \Pi \quad \text{in}~\mathcal{M}((0,T)\times \TT),
\end{equation*}
whereas
\begin{equation*}
\Lambda_\ep(\rho_\ep) \leq C \ep \mathbf{1}_{\{\rho_\ep \leq M^0\}} 
+ C \ep^{1- \frac{\beta-1}{\gamma}} \big(p_\ep(\rho_\ep) \big)^{\frac{\beta-1}{\gamma}} \mathbf{1}_{\{\rho_\ep > M^0\}} 
\end{equation*}
with $\gamma > \beta-1$, so that $\Lambda_\ep(\rho_\ep)$ converges strongly to $0$ in $L^{\frac{\gamma}{\beta-1}}((0,T)\times \TT)$ (and thus in $L^2(0,T;L^{\frac{6}{5}}(\TT))$).
Passing to the limit in~\eqref{eq:Lambda_ep-2}, we have then
\[
\Pi = 0.
\]
We thus conclude that the triplet $(\rho, u, p)$ is a global weak solution of 
\begin{subnumcases}{\label{eq:xxx3}}
\partial_t \rho + \Div (\rho u) = 0 \\
\nabla p - 2 \Div(\mu \D(u)) + ru = f \\
0 \leq \rho \leq 1, ~\mathrm{spt} \ p \subset \{\rho = 1\},~ p \geq 0.
\end{subnumcases}


\subsection{Case $1 <\gamma < \beta - 1$}

In this case we expect the activation of the memory effects in the limit $\ep \rightarrow 0$.

\subsubsection{Existence of weak solutions at $\ep$ fixed}
First we observe that Lemma~\ref{lem:controls-1} still holds in this case, so that
\begin{equation*}
\big(\Lambda_{\ep,\delta}(\rho_{\ep,\delta})\big)_\delta \quad \text{is bounded in} \quad L^\infty\big(0,T;L^1(\TT)\big)
\end{equation*} 
and, since $\gamma < \beta-1$,
\[
\big(p_{\ep,\delta}(\rho_{\ep,\delta})\big)_\delta \quad \text{is bounded in} \quad L^\infty(0,T;L^{\frac{\beta-1}{\gamma}}(\TT)).
\]
With this control of the pressure and the energy estimate, we deduce from Lemma~\ref{lambdadivu} that
\begin{equation}\label{eq:lambdadivu-3}
\big(\lambda_{\ep,\delta}(\rho_{\ep,\delta}) \Div u_{\ep,\delta}\big)_\delta ~ \text{is bounded in}~
L^2\big(0,T;L^{\min\{2,\frac{\beta-1}{\gamma}\}}(\TT)\big).
\end{equation}

\bigskip
\noindent
\textit{Uniform upper bound on the density.}
The same arguments as before show that
\begin{equation*}
\|\rho_{\ep,\delta}\|_{L^\infty_{t,x}} \leq \bar{\rho} < \infty.
\end{equation*}
and
\begin{equation*}
\mathrm{meas} \,\big\{ x\in \TT , \ \rho_{\ep,\delta}(t,x) \geq 1-\delta \big\}  \leq C(\ep) \ \delta^{\beta-1} .
\end{equation*}

\bigskip
\noindent
\textit{Additional integrability on $\Lambda_{\ep,\delta}$.}
\begin{lem}\label{lem:controls-3}Let $(\rho_{\varepsilon, \delta}, u_{\varepsilon,\delta})$ be a global weak solution of the compressible Brinkman system~\eqref{eq:semi-stat-delta} with $1 <\gamma < \beta-1$.
	If initially $\|\Lambda_\ep(\rho^0_\ep)\|_{L^2} \leq C$, then there exist $C^1>0$ independent of $\delta,\ep$, and $C^2_\ep> 0$ independent of $\delta$, such that
	\begin{equation}
	\|\Lambda_{\ep,\delta}(\rho_{\ep,\delta})\|_{L^\infty L^2} + \|p_{\ep,\delta}(\rho_{\ep,\delta})\|_{L^\infty L^{\frac{2(\beta-1)}{\gamma}}}
	\leq C^1,
	\end{equation}
	\begin{equation}
	\|\lambda_{\ep,\delta}(\rho_{\ep,\delta})\|_{L^\infty L^{\frac{2(\beta-1)}{\beta}}} 
	\leq C^2_\ep.
	\end{equation}
\end{lem}

\medskip
\begin{proof} We can reproduce exactly the same estimate on $\big(\Lambda_{\ep,\delta}(\rho_{\ep,\delta})\big)^{2}$ by multiplying Equation~\eqref{eq:Lambda_delta} by $\Lambda_{\ep,\delta}(\rho_{\ep,\delta})$.
We obtain
\begin{equation*}
\sup_{t \in [0,T]}\int_{\TT}{\big(\Lambda_{\ep,\delta}(\rho_{\ep,\delta})\big)^{2}} 
+ \dfrac{2\beta -1}{\beta-1} \int_{\TT}{\Lambda_{\ep,\delta}(\rho_{\ep,\delta}) p_{\ep,\delta} (\rho_{\ep,\delta})} 
\leq C + C\int_{\TT}{\big(\Lambda_{\ep,\delta}(\rho_{\ep}^0)\big)^{2}}.
\end{equation*}
Hence
\begin{equation*}\label{eq:Lambdadelta3}
\big(\Lambda_{\ep,\delta}(\rho_{\ep,\delta})\big)_\delta ~ \text{is bounded in} ~ L^\infty(0,T;L^2(\TT))
\end{equation*}
and consequently
\begin{equation*}\label{eq:pdelta3}
\big(p_{\ep,\delta}(\rho_{\ep,\delta})\big)_\delta  ~ \text{is bounded in} ~ L^\infty(0,T;L^{\frac{2(\beta-1)}{\gamma}}(\TT)),
\end{equation*}
\begin{equation*}\label{eq:lambdadelta3}
\big(\lambda_{\ep,\delta}(\rho_{\ep,\delta})\big)_\delta  ~ \text{is bounded in} ~ L^\infty(0,T;L^{\frac{2(\beta-1)}{\beta}}(\TT)).
\end{equation*}
This last control on $\lambda_\ep(\rho_\ep)$ being not uniform in $\ep$.
\end{proof}

\begin{remark}\label{rk:control-lambda}
	We assumed in this case that $1 <\gamma<\beta-1$, so that $\beta > 2$ and
	\[
	\dfrac{2(\beta-1)}{\beta} > 1.
	\]
	The bulk viscosity coefficient $\big(\lambda_{\ep,\delta}(\rho_{\ep,\delta})\big)_\delta$ is then bounded in $L^1((0,T) \times \TT)$.
\end{remark}

\bigskip
\noindent
\textit{Limit $\delta \rightarrow 0$.}
We can pass to the limit in the weak formulation of the equations exactly in the same way as before, showing first the strong convergence of the density and using the control of $\sqrt{\lambda_{\ep,\delta}(\rho_{\ep,\delta})}$ in $L^2\big((0,T)\times \TT\big)$ to identify the limit of the bulk viscosity term:
\[
\lambda_{\ep,\delta}(\rho_{\ep,\delta})\Div u_{\ep,\delta}
\longrightarrow 
\lambda_\ep(\rho_\ep)\Div u_\ep ~ \text{weakly in}~ L^1((0,T)\times \TT).
\]

\subsubsection{Limit $\ep \rightarrow 0$}
Thanks to the uniform controls of the singular quantities, we ensure the following convergences
\begin{equation}
\Lambda_\ep(\rho_\ep) \longrightarrow \Lambda \quad \text{weakly-* in} ~ L^\infty(0,T;L^2(\TT)),
\end{equation}
\begin{equation*}
\lambda_\ep(\rho_\ep) \Div u_\ep \longrightarrow -\Pi \quad \text{weakly in} ~
\mathcal{M}((0,T)\times \TT).
\end{equation*}
In addition,
	\[
	\Lambda_\ep(\rho_\ep) \mathbf{1}_{\{\rho_\ep \leq 1-\ep^{1/\beta} \}} = \dfrac{\ep}{\beta-1}\dfrac{\rho_\ep^\beta}{\ep^{\frac{\beta-1}{\beta}}} \leq  \dfrac{\ep^{1/\beta}}{\beta-1}
	~\underset{\ep \rightarrow 0}{\longrightarrow} ~ 0.  
	\]
Hence
\[
\mathrm{spt}\, \Lambda \subset \{ \rho = 1\},
\]
or written differently, the product $\rho \Lambda$ making sense almost everywhere,
\begin{equation}
(1-\rho) \Lambda = 0.
\end{equation}
On the other hand, we have $\gamma < \beta-1$, which yields
\begin{equation*}
p_\ep(\rho_\ep) \longrightarrow 0 \quad\text{strongly in} ~ L^1((0,T)\times \TT).
\end{equation*}
Compared to the previous case, the limit $\Lambda$ is not $0$ and we close the system by passing to the limit in 
\[
\partial_t \Lambda_\ep(\rho_\ep) + \Div(\Lambda_\ep(\rho_\ep) u_\ep) = - \lambda_\ep(\rho_\ep)\Div u_\ep.
\]
Here, $\Lambda_\ep$ is controlled in $L^\infty(0,T; L^2(\TT))$ and $\|\partial_t \Lambda_\ep\|_{L^1(W^{-1,1})}\leq C$, while $u_\ep$ is bounded in $L^2(0,T;(H^1(\TT))^3)$.
We can then pass to the limit in the product $\Lambda_\ep(\rho_\ep) u_\ep$ thanks to Lemma~\ref{lem:compensated_compactness} and get the limit equation
\[
\partial_t \Lambda + \Div(\Lambda u) = \Pi.
\]
We have thus justified the activation of memory effects in the congestion limit.
The tuple $(\rho, u, \pi,\Lambda)$ is finally a global weak solution of
\begin{subnumcases}{\label{eq:xxx}}
\partial_t \rho + \Div (\rho u) = 0 \\
\nabla \Pi - 2 \Div(\mu \D(u)) + r u  = f \\
\partial_t \Lambda + \Div(\Lambda u) = \Pi\\
0 \leq \rho \leq 1, ~(1-\rho)\Lambda = 0,~ \Lambda \geq 0
\end{subnumcases}


\subsection{Case $1<\gamma = \beta-1$}
The estimates remain mostly unchanged compared to the previous case $\gamma < \beta-1$.
Nevertheless, in the limit $\ep \rightarrow 0$, we do not have the convergence of $p_\ep(\rho_\ep)$ to $0$ anymore and $p_\ep(\rho_\ep)$ converges now to $p = (\beta-1) \Lambda$.
The limit system then writes
\begin{subnumcases}{\label{eq:xxx2}}
\partial_t \rho + \Div (\rho u) = 0 \\
(\beta-1)\nabla \Lambda + \nabla \Pi - 2 \Div(\mu \D(u)) + ru = f \\
\partial_t \Lambda + \Div(\Lambda u) = \Pi\\
0 \leq \rho \leq 1, ~(1-\rho)\Lambda = 0,~ \Lambda \geq 0
\end{subnumcases}


\bigskip

\noindent {\bf Acknowledgments.} C. Perrin is supported by a PEPS project « Jeunes chercheuses et jeunes chercheurs ».
 Part of this work was carried out when D. Bresch was inviting in Prague by the GA\v CR project P201-16-032308 and RVO 67985840. \v S. Ne\v casov\' a is supported by GA\v CR project P201-16-032308 and RVO 67985840. D. Bresch is also supported by the ANR project FRAISE.
\newpage

\end{document}